\documentclass[12pt,a4paper]{article}
\usepackage[utf8]{inputenc}
\usepackage{cite}
\usepackage[a4paper,top=2cm,bottom=2cm,left=2.5cm,right=2.5cm,marginparwidth=1.75cm]{geometry}
\usepackage{pgfplots}
\pgfplotsset{compat=1.15}
\usepackage{mathrsfs}
\usetikzlibrary{arrows}
\usepgfplotslibrary{fillbetween,decorations.softclip}
\pgfplotsset{compat = newest}
  
\usepackage{amsmath}
\usepackage{graphicx}
\usepackage[colorlinks=true, allcolors=blue]{hyperref}
\usepackage{hyperref}
\usepackage{orcidlink}
\usepackage[title]{appendix}
\usepackage{mathrsfs}
\usepackage{amsfonts}
\usepackage{booktabs} 
\usepackage{caption}  
\usepackage{threeparttable} 
\usepackage{algorithm}
\usepackage{algorithmicx}
\usepackage{algpseudocode}
\usepackage{listings}
\usepackage{enumitem}
\usepackage{chngcntr}
\usepackage{booktabs}
\usepackage{lipsum}
\usepackage{subcaption}
\usepackage{authblk}
\usepackage[T1]{fontenc}    
\usepackage{csquotes}       
\usepackage{diagbox}
\usepackage{amsthm}
\theoremstyle{definition}
\newtheorem{definition}{Definition}[section]
\theoremstyle{remark}
\newtheorem*{remark}{Remark}
\newtheorem{theorem}{Theorem}[section]
\newtheorem{proposition}[theorem]{Proposition}
\newtheorem{lemma}[theorem]{Lemma}

\newtheorem{claim}[theorem]{Claim}
\newtheorem{corollary}{Corollary}[section]
\allowdisplaybreaks

\usepackage{setspace}
\usepackage{titlesec}
\titleformat{\section} 
  {\normalfont\Large\bfseries}{\thesection.}{1em}{}
  
\usepackage{float}   
\usepackage{caption} 



\title{A uniqueness result in the inverse problem for the anisotropic Schr\"odinger type equation from local measurements}

\author[1]{Niall Donlon}
\author[2]{Romina Gaburro}
    \affil[1]{\footnotesize Department of Mathematics and Statistics, University of Limerick, Ireland.  \newline Email: niall.donlon@ul.ie}
\affil[2]{\footnotesize Department of Mathematics and Statistics, Health Research Institute, University of Limerick, Ireland. \newline Email: romina.gaburro@ul.ie}

\date{}  

\begin{document}
\maketitle

\begin{abstract}
We consider the inverse boundary value problem of the simultaneous determination of the coefficients $\sigma$ and $q$ of the equation $-\mbox{div}(\sigma \nabla u)+qu = 0$ from knowledge of the so-called Neumann-to-Dirichlet map, given locally on a non-empty curved portion $\Sigma$ of the boundary $\partial \Omega$ of a domain $\Omega \subset \mathbb{R}^n$, with $n\geq 3$. We assume that $\sigma$ and $q$ are \textit{a-priori} known to be a piecewise constant matrix-valued and scalar function, respectively, on a given partition of $\Omega$ with curved interfaces. We prove that $\sigma$ and $q$ can be uniquely determined in $\Omega$ from the knowledge of the local map.
\end{abstract}


\section{Introduction}
We investigate the inverse problem of the simultaneous determination of the coefficients in a Schr\"odinger type equation 
\begin{equation}\label{eq1}
    -\mbox{div}(\sigma \nabla u) + qu =0, \qquad \text{in\quad $\Omega$,}
\end{equation}
where $\Omega\subset\mathbb{R}^n$ is a domain, with $n\geq 3$, from the knowledge of a local version of the Neumann-to-Dirichlet (N-D) map
   \[\mathcal{N}_{\sigma, q} : \sigma \nabla u \cdot \nu|_{\partial \Omega} \in H^{-\frac{1}{2}}(\partial \Omega) \longrightarrow u|_{\partial \Omega} \in H^{\frac{1}{2}}(\partial \Omega),\]
for every $u \in H^1(\Omega)$ solution to \eqref{eq1}, where $\nu$ denotes the outer unit normal to $\partial \Omega$. A rigorous definition of the N-D map and its local version are provided in Section \ref{ND map section}. We assume that $\sigma$ is a symmetric matrix-valued function satisfying the uniform ellipticity condition
\[ \lambda^{-1}|\xi|^2 \leq \sigma(x) \xi \cdot \xi \leq \lambda |\xi|^2,\quad
\text{for\:a.e.\quad $x \in \Omega$,\quad for every $\xi \in \mathbb{R}^n$} \]
and $q\geq 0$ is a scalar-valued function, with $q>0$ on a subset of $\Omega$ of positive Lebesgue measure. Moreover, given a known partition of $\Omega$, $\{D_j\}_{j=1}^N$, $\sigma$ and $q$ are both constant on each $D_j$, for $j=1,\dots , N$. 
The inverse problem addressed here is the simultaneous identification of $\sigma$ and $q$ from a local N-D map, $\mathcal{N}_{\sigma, q}$.
This model includes a large class of inverse problems. When $q=0$, $\sigma$ represents the conductivity in a conductor $\Omega$. Here the determination of $\sigma$ from the Dirichlet-to-Neumann (D-N) or, similarly, the N-D map, is the celebrated Calder\'on’s problem \cite{calderon2006inverse}. While the issue of uniqueness for Calder\'on’s problem (when $q=0$) is well studied in the isotropic case and almost considered to be settled \cite{alessandrini1990singular} \cite{kohn1984determining}  \cite{kohn1984identification} \cite{sylvester1987global} (see also \cite{caro2016global} \cite{haberman2015uniqueness} \cite{haberman2013uniqueness} for $n\geq 3$, and \cite{astala2006} \cite{brown1997uniqueness} for $n = 2$), the issue of the unique determination of an anisotropic conductivity $\sigma$ from a boundary map (D-N or N-D), is still an open problem even in the case when measurements are modeled by a global boundary map.\\
We recall Tartar's observation \cite{kohn1984identification} that any diffeomorphism $\psi$ of $\Omega$ which keeps the boundary points fixed, leaves the D-N map unchanged, while $\sigma$ is modified via its push-forward under $\psi$
\[ \psi^*\sigma := \frac{(D\psi)^T \sigma (D\psi)}{|\det (D\psi)|} \circ \psi^{-1}.\]
Hence, it has been shown that one can determine $\sigma$ up to a diffeomorphism that leaves the boundary points fixed (see \cite{Belishev2003} \cite{As-La-P2005} \cite{lassas2001determining} \cite{lassas2003dirichlet} \cite{Lee-U1989} \cite{nachman1996global} \cite{sylvester1990anisotropic} and references therein). Alternatively, assuming that the anisotropic structure of $\sigma$ is \textit{a-priori} known to depend on a finite number of spatially dependent scalar functions, one can restore stability, the continuous dependence of $\sigma$ upon the relevant map (D-N or N-D) at the boundary (\cite{A-G2001} \cite{A-G2009} \cite{G-L2009} \cite{kohn1984identification} \cite{L1997}) and in the interior (\cite{F-G-S2021} \cite{foschiatti2025recovering}) \cite{G-S2015}).\\
Here, we extend the uniqueness result obtained in \cite{Al2017} for $q=0$ to the case when $q\neq 0$ and non-negative. In \cite{Al2017}, which has a different flavor compared to the above mentioned results, the authors proved that an anisotropic conductivity that ia \textit{a-priori} known to be a piecewise constant matrix on a given partition of $\Omega$ with curved interfaces, can be uniquely determined from the knowledge of a local N-D map, localised on a curved portion $\Sigma$ of $\partial\Omega$. Our main result shows that in the case when the forward problem is \eqref{eq1}, with $q\neq 0$ and non-negative, the knowledge of the local N-D map, allows also for the simultaneous unique determination of $\sigma$ and $q$ in $\Omega$, in a setting analogous to that considered in \cite{Al2017} (Theorem \ref{Main theorem of uniqueness}). We also show that when the interfaces of the partition of $\Omega$ considered in Theorem \ref{Main theorem of uniqueness} are $C^{1, \alpha}$ boundaries of an unknown nested family of subdomains of $\Omega$, then the knowledge of the local N-D map allows to determine $\sigma$, $q$ in $\Omega$, together with the interfaces of the unknown partition of $\Omega$ (Corollary \ref{Nested domain theorem}), extending the results obtained in \cite{alessadrini2018eit} where the case $q=0$ was considered (see also \cite{Carstea2018} in the context of elastostatics with anisotropic elasticity tensor).
\\
As an application of the case $q\geq 0$ treated here, \eqref{eq1} models the propagation of light through a body $\Omega$. In this case, $\sigma$ and $q$ represent the diffusion tensor and absorption coefficients, respectively, in the inverse problem known as Diffuse Optical Tomography (DOT) \cite{applegate2020recent} \cite{arridge1999optical} \cite{arridge2009optical} \cite{curran2024singular} \cite{curran2023timeharmoic}. The problem of uniqueness in DOT for globally smooth coefficients has been studied extensively. If $\sigma$ is smooth and isotropic ($\sigma=\gamma I$, where $\gamma$ is a scalar-valued function and $I$ denotes an $n\times n$ identity matrix), one can set $v := \sqrt{\sigma }u$, to transform \eqref{eq1} into
\[ - \Delta \nu + \eta v = 0,  \]
with an \textit{effective absorption}
\[ \eta := \frac{\Delta \sqrt{\sigma}}{\sqrt{\sigma}} + \frac{q}{\sigma}. \]
If $\sigma =1$ in a neighborhood of $\partial \Omega$, then $u$ and $v$ have the same boundary values on $\partial \Omega$. Hence, the Cauchy data can only contain information about the effective absorption $\eta$ from which, generally, one cannot determine $\sigma$ and $q$, showing that, in general,  the inverse problem in DOT is not uniquely solvable \cite{arridge1998nonuniqueness} for both $\sigma$ and $q$. However, in the isotropic case, Harrach proved uniqueness when $\gamma$ and $q$ are assumed to be piecewise constant and piecewise analytic, respectively, \cite{harrach2012simultaneous} via the use of localised potentials \cite{harrach2009uniqueness}.
In \cite{foschiatti2024lipschitz}, global Lipschitz stability has been restored for a Schr\"odinger type equation in terms of the Cauchy data, for the case when $\sigma=\gamma A$, $A$ is a known $C^{1,1}$ matrix-valued function, $\gamma$ is an unknown piecewise affine scalar function, and $q$ is piecewise affine on a given partition of $\Omega$ and no sign, nor spectrum condition on $q$ is assumed (see also \cite{francini2023propagation} for the related inverse problem of estimating the size of an inclusion for \eqref{eq1} with complex discontinuous coefficients).\\
Here, we consider the inverse problem of the simultaneous unique determination of piecewise constant $\sigma$ (symmetric matrix-valued function satisfying a uniform ellipticity condition) and $q\neq0$ and non-negative in \eqref{eq1}, from a local D-N map localised on a curved $C^{1,\alpha}$ open portion $\Sigma\subset\partial\Omega$ (the precise definitions of $C^{1,\alpha}$ and curved open portion of $\partial\Omega$ are given in section \ref{sec2}). Specifically, we assume that $\Omega$ is covered by a known finite partition $\{D_j\}_{j=1}^N$ of non-overlapping connected subdomains, having interfaces that contain a portion which is curved and $C^{1,\alpha}$. $q\neq 0$ and non-negative is also assumed to be piecewise constant on the given partition. Our approach follows that in \cite{Al2017}, where the case $q=0$ was considered. Denoting by $L$ the differential operator in \eqref{eq1}, we have that 
\[{\frac{1}{\sqrt{\det g}}}L=-\Delta_g + {\frac{1}{\sqrt{\det g}}}q,\] 
where $g$ is related to $\sigma$ via $g= (\text{det}\sigma)^{\frac{1}{n-2}} \sigma^{-1}$ \cite{berger1971spectre} \cite{uhlmann2009electrical}, when $n\geq 3$, which is the case considered here. Under the \textit{a-priori} assumption that $\sigma\in C^{\alpha}(\mathcal{U})$, where $\mathcal{U}$ is a neighborhood of $\Sigma$, we can uniquely determine the tangential part of $g$ on $\Sigma$; hence, the $(n-1) \times (n-1)$ sub-matrix of $g$ from knowledge of the local N-D map. When $g$ is constant in $\mathcal{U}$, by taking advantage of the non flatness of $\Sigma$, we have sufficient information to recover $g$ entirely on $\mathcal{U}$ and hence determine $\sigma$ on $\mathcal{U}$ as well. This allows for the unique determination of $\sigma$ on the subdomain of the partition, say $D_1$, which boundary intersects $\Sigma$ on a $C^{1,\alpha}$ (and curved) portion. Once $\sigma$ is identified in $D_1$, we combine the method of singular solutions introduced in \cite{alessandrini1990singular} together with the asymptotic behavior of the Neumann Kernel of \eqref{eq1} in $\Omega$ near its pole, chosen on $\Sigma$. This allows for the unique determination of $q$ in $D_1$ from knowledge of $\sigma$ and the local N-D map. 
The proof is then completed by an induction argument and the use of the unique continuation property that allow to uniquely determine, domain by domain, both $\sigma$ and $q$ within $\Omega$.\\
The paper is organized as follows. Our main assumptions and results, Theorem \ref{Main theorem of uniqueness} and Corollary \ref{Nested domain theorem} are presented in section \ref{sec2}. The proof of Theorem \ref{Main theorem of uniqueness} is given in section \ref{sec3}. In section \ref{sec3}, we show how to uniquely determine $\sigma$ in a neighborhood $\mathcal{U}$ of $\Sigma$, where the local N-D map has been localised, in terms of the local map and our \textit{a-priori} assumptions. This allows to identify $\sigma$ in the domain $D_1$ of our partition, which boundary intersects $\Sigma$ on a (curved and $C^{1, \alpha}$) portion; in section \ref{sec4}, the so-called \textit{potential} $q$ is then uniquely determined in $D_1$ as well by knowledge of the N-D map, $\sigma$ and the \textit{a-priori} assumptions. The unique determination of $\sigma$ and $q$ within $\Omega$ from knowledge of the local N-D map is then proved in section \ref{sec5}: the main argument here is inductive and take advantage of the unique continuation property of solutions to \eqref{eq1} in our setting. An Appendix is then devoted to the proof of Corollary \ref{Nested domain theorem} and includes also a more or less trivial extension of it.


\section{Assumptions and main results}\label{sec2}

We assume throughout that $\Omega \subset \mathbb{R}^n$, $n\geq 3$ is a bounded domain with Lipschitz boundary, as per definition \ref{LipschitzBoundary} below. For $n\geq3$, a point $x\in\mathbb{R}^n$ will be denoted by $x=(x',x_n)$, where $x'\in\mathbb{R}^{n-1}$ and $x_n\in\mathbb{R}$. Moreover, given a point $x\in\mathbb{R}^n$, we denote with $B_r(x),B'_r(x)$ the open balls in $\mathbb{R}^n, \mathbb{R}^{n-1}$, respectively, centered at $x$ with radius $r$ and by $Q_r(x)$ the cylinder $B'_r(x') \times (x_n-r, x_n +r)$. We also denote $\mathbb{R}^n_+ = \{(x',x_n)\in \mathbb{R}^n | \quad x_n >0\}$, $B_r = B_r(0)$, $Q_r = Q_r(0)$ and $B_r^+ = B_r \cap \mathbb{R}^n_+$.\\
\begin{definition}\label{LipschitzBoundary}
We say that the boundary of $\Omega$, $\partial \Omega$, is of Lipschitz class 
with constants if there is $r_0>0$ such that for any $P\in\partial\Omega$, there exists a rigid transformation of coordinates under which we have $P=0$ and 
\begin{equation}
\Omega \cap Q_{r_0} = \left\{ (x', x_n) \in Q_{r_0} \; \vert \; x_n > \phi(x') \right\},
\end{equation} 
where $\phi$ is a Lipschitz function on $B_{r_0}'$ satisfying $\phi(0) = 0$. 
\end{definition}

\begin{definition}\label{function of class C^1 alpha}
 Given $\alpha \in (0,1)$, we say that an open portion $\Sigma$ of $\partial\Omega$ is of class $C^{1,\alpha}$ if for any $P\in \Sigma$ there exists a rigid transformation of coordinates under which we have $P=0$ and 
    \begin{equation}\label{phi 1}
    \Omega \cap Q_{r_0} = \{ (x', x_n)\in Q_{r_0} \:|\: x_n > \varphi (x')\},
    \end{equation}
    for some $r_0 >0$, where $\varphi$ is a $C^{1,\alpha}$ function on $B'_{r_0}$ satisfying 
    \begin{equation}\label{phi 2}
    \varphi(0) = |\nabla_{x'}\varphi(0)| = 0.
    \end{equation}
\end{definition}


\begin{definition}\label{nonflatness}
   If $\Sigma$ is a $C^{1,\alpha}$ open portion as per definition \ref{function of class C^1 alpha}, we say that $\Sigma$ is non-flat (also the function $\varphi$) if there exist three points $y_1, y_2, y_3 \in \Sigma$, such that
\begin{equation}
   \nu(y_i) \cdot \nu (y_j) < 1,\quad\textnormal{for}\quad i,j=1,2,3,\quad\textnormal{with}\quad i\neq j.
       \end{equation}
\end{definition}
Once and for all, definitions \ref{LipschitzBoundary}-\ref{nonflatness} will be adopted throughout the entire manuscript when referring to boundaries of domains that are Lipschits, $C^{1,\alpha}$ and non-flat, respectively.\\
We fix an open portion $\Sigma\subset\partial\Omega$, where our measurements are going to be localised, as follows. We assume that there is a point $y \in\partial\Omega$ and a neighbourhood $\mathcal{U}$ of $y$, such that up to a rigid transformations of coordinates,  $y=0$, and 
\begin{equation}\label{neighbourhood}
    \Sigma = \partial \Omega \cap \mathcal{U},
\end{equation}
is a non-flat portion of $\partial\Omega$ of class $C^{1,\alpha}$, for some $\alpha\in (0,1)$.


\subsection{The Neumann-to-Dirichlet map}\label{ND map section}
Let $\sigma \in L^\infty (\Omega, Sym_n)$, where $Sym_n$ denotes the class of $n \times n$ symmetric real valued matrices, satisfy the uniform ellipticity condition 
    \begin{equation}\label{ellipticity condition 0}
\lambda^{-1}|\xi|^2 \leq \sigma(x) \xi \cdot \xi \leq \lambda |\xi|^2,\quad
\text{for\:a.e.\quad $x \in \Omega$,\quad for every $\xi \in \mathbb{R}^n$,}
    \end{equation}
and $q\in L^\infty (\Omega)$ be a non-negative scalar-valued function in $\Omega$ such that $q>0$ on a subset of $\Omega$ of positive Lebesgue measure (we will refer to $q$ as the \textit{potential}). Under these general assumptions on $\sigma$ and $q$, we define the Neumann-to-Dirichlet map for the operator
    \begin{equation}\label{L}
    L= -\mbox{div}(\sigma (x) \nabla \cdot) + q(x),\quad \text{in $\Omega$}.
\end{equation} 
\begin{definition}\label{ND map}
   The Neumann-to-Dirichlet (N-D) map corresponding to $\sigma, q$ in \eqref{L},
    \[\mathcal{N}_{\sigma, q} : H^{-\frac{1}{2}}(\partial \Omega) \longrightarrow H^{\frac{1}{2}}(\partial \Omega),\]
    is the operator $\mathcal{N}_{\sigma, q}$ defined by
    \begin{equation}
        \langle \psi,  \mathcal{N}_{\sigma, q} \psi \rangle = \int_\Omega \sigma(x) \nabla u(x) \cdot \nabla u (x) + q(x)u(x) u(x) \textnormal{d}x,
    \end{equation}
    for every $\psi \in {H^{-\frac{1}{2}}(\partial \Omega) }$, where $u \in H^{1}(\Omega)$ is the weak solution to the Neumann problem
     \begin{equation}
        \Bigg\{\begin{array}{llll}
-\mbox{div}(\sigma(x)\nabla u(x)) +q(x) u(x) =0, \quad \text{in} \quad \Omega, \\
\sigma \nabla u \cdot \nu|_{\partial\Omega}=\psi, \quad \text{on} \quad \partial\Omega. \\
        \end{array}
    \end{equation}
\end{definition}
    
We set $\Delta = \partial \Omega \backslash \overline{\Sigma}$, where $\Sigma$ has been defined above, define
\[H^{\frac{1}{2}}_{co}(\Delta)= \{ f \in H^{\frac{1}{2}}(\partial \Omega)\quad\vert\quad\textnormal{supp}(f) \subset \Delta\},\]
and denote by $H^{\frac{1}{2}}_{00}(\Delta)=\overline{H^{\frac{1}{2}}_{co}(\Delta)}$, where the closure is with respect to the $H^{\frac{1}{2}}(\partial \Omega)$ norm. We also introduce
\begin{equation}
    H^{-\frac{1}{2}}(\Sigma) = \{ \psi \in   H^{-\frac{1}{2}}(\partial \Omega)\quad\vert\quad \langle \psi, f\rangle =0, \quad \text{for any $f\in H^{\frac{1}{2}}_{00}(\Delta)$} \},
\end{equation}
that is the space of distributions $\psi \in H^{-\frac{1}{2}}(\partial \Omega)$ which are supported in $\overline{\Sigma}$.
\begin{definition}
    The local N-D map corresponding to $\sigma$, $q$ in \eqref{L} and $\Sigma$ is the operator
    \[\mathcal{N}^{\Sigma}_{\sigma, q} : { H^{-\frac{1}{2}}(\Sigma)} \longrightarrow ({ H^{-\frac{1}{2}}(\Sigma)})^* \subset {H^{\frac{1}{2}}(\partial \Omega)},\]
    given by
    \begin{equation}
       \langle \mathcal{N}^{\Sigma}_{\sigma, q} \varphi, \psi \rangle = \langle \mathcal{N}_{\sigma, q} \varphi, \psi \rangle, 
    \end{equation}
    for every $\varphi, \psi \in {H^{-\frac{1}{2}}(\Sigma)}$. 
\end{definition}
 Given $\sigma^{(i)}\in L^\infty (\Omega, Sym_n)$, satisfying \eqref{ellipticity condition 0} and $q^{(i)}\in L^{\infty}(\Omega)$ non-negative scalar-valued functions that are positive on a subset of $\Omega$ of positive Lebesgue measure, for $i=1,2$, we recall Alessandrini's identity
    \begin{equation}\label{Alessandrini}
    \langle (\psi_1, (\mathcal{N}^\Sigma_{\sigma^{(2)}, q^{(2)}} - \mathcal{N}^\Sigma_{\sigma^{(1)} , q^{(1)}} )\psi_2 \rangle = \int_\Omega \big(\sigma^{(1)} - \sigma^{(2)}\big)\nabla u_1 \cdot \nabla u_2  + (q^{(1)}  - q^{(2)})u_1\: u_2,
\end{equation}
for any $\psi_i \in {H^{-\frac{1}{2}}(\Sigma)}$ and $u_i \in H^1(\Omega)$ being the weak solution to the Neumann problem
     \begin{equation}
        \Bigg\{\begin{array}{llll}
-\mbox{div}(\sigma^{(i)}(x)\nabla u_i(x)) +q^{(i)}(x) u_i(x) =0, \quad \text{in} \quad \Omega, \\
\sigma^{(i)} \nabla u_i \cdot \nu|_{\partial\Omega}=\psi_i, \quad \text{on} \quad \partial\Omega, \\
        \end{array}
    \end{equation}
for $i=1,2$. 


\subsection{Unique determination of $\sigma$ and $q$ in $\Omega$.}\label{sec general partition}
We consider a partition of $\Omega$, $\{D_j\}^N_{j=1}$, with $N\in (\mathbb{N} \backslash\{0\})$, such that 
        \[\overline{\Omega} = \bigcup^N_{j=1} \overline{D}_j,\]
        where $D_j$, $j=1,\dots,N$ are known open sets of $\mathbb{R}^n$,  satisfying the following properties (see Figure 1):
    
    \begin{enumerate}
        \item $\;$ $D_j$, $j=1,\dots,N$ are connected, pairwise non overlapping and the boundary of $\partial D_j$ is of Lipschitz class, for $j=1,\dots,N$. 
        \item $\;$ There is one domain, $D_1$ such that $\partial D_1 \cap \Sigma \neq \emptyset$  contains a non-flat $C^{1,\alpha}$ portion $\Sigma_1$.
        \item$\;$  For every $i\in\{2,\dots,N\}$, there exits $j_1,\dots, j_K \in \{1,\dots, N\}$ such that
        \[D_{j_1} = D_1, \quad D_{j_K} = D_i,\]
        and such that
        \[\Bigg( \bigcup^K_{k=1} \overline{D}_{j_k}\Bigg)^\circ \quad \text{and} \quad \Omega \backslash \Bigg( \bigcup^K_{k=1} \overline{D}_{j_k}\Bigg)\]
        are Lipschitz domains. 
        We also assume that for every $k=2,\dots,K$, $\partial D_{j_k} \cap \partial D_{j_{k-1}}$ contains a non-flat $C^{1,\alpha}$ portion $\Sigma_k$ such that $\Sigma_k \subset \Omega$. Moreover, we assume there exists $P_k \in \Sigma_k$, for every $k =2, \dots K$ and a rigid transformation of coordinates under which we have that $P_k =0$ and 
        \begin{equation}
            \begin{split}
                \Sigma_k \cap Q_{\frac{r_0}{3}} & = \{ x \in Q_{\frac{r_0}{3}} | \; x_n = \varphi_k(x')\},\\
                 D_{j_k} \cap Q_{\frac{r_0}{3}} & = \{ x \in Q_{\frac{r_0}{3}} | \; x_n > \varphi_k(x')\},\\
                  D_{j_{k-1}} \cap Q_{\frac{r_0}{3}} & = \{ x \in Q_{\frac{r_0}{3}} | \; x_n < \varphi_k(x')\},
            \end{split}
        \end{equation}
        where $\varphi_k$ in a non-flat $C^{1, \alpha}$ function on $B'_{r_0}$ satisfying $\varphi_k(0) = |\nabla \varphi_k(0)| = 0$.
    \end{enumerate}


   

\begin{figure}[!h] 
\centering
\begin{tikzpicture}[x=0.6pt,y=0.6pt,yscale=-.45,xscale=.45]

\draw [color={rgb, 255:red, 0; green, 0; blue, 0 }  ,draw opacity=1 ]   (438.98,124.68) .. controls (462.04,140.31) and (645.65,99.51) .. (601.36,351.28) .. controls (557.08,603.05) and (98.52,550.09) .. (61.61,342.59) .. controls (24.7,135.1) and (185.25,133.37) .. (208.31,121.21) ;
\draw    (217,172.5) .. controls (253.91,146.45) and (398.38,150.73) .. (432.52,168.96) ;
\draw    (227.23,224.09) .. controls (264.13,198.05) and (394.69,200.22) .. (428.83,218.45) ;
\draw    (235.53,267.5) .. controls (272.44,241.45) and (401.61,248.4) .. (423.29,266.2) ;
\draw [color={rgb, 255:red, 0; green, 0; blue, 0 }  ,draw opacity=1 ][line width=0.75]    (271,466.5) .. controls (299.92,455.95) and (365,456.5) .. (399,467.5) ;
\draw  [color={rgb, 255:red, 255; green, 255; blue, 255 }  ,draw opacity=1 ][fill={rgb, 255:red, 255; green, 255; blue, 255 }  ,fill opacity=1 ] (376.14,406.22) .. controls (363.03,405.1) and (342.29,404.38) .. (318.98,404.38) .. controls (303.13,404.38) and (288.51,404.71) .. (276.68,405.28) -- (319.16,409.01) -- cycle ;
\draw [color={rgb, 255:red, 208; green, 2; blue, 27 }  ,draw opacity=1 ]   (208.31,121.21) .. controls (277.51,73.46) and (336.56,43.08) .. (438.98,124.68) ;
\draw [color={rgb, 255:red, 0; green, 0; blue, 0 }  ,draw opacity=1 ][line width=0.75]    (254.57,374.85) .. controls (283.48,364.3) and (368.6,361.2) .. (408.6,377.6) ;
\draw    (208.31,121.21) -- (271,466.5) ;
\draw    (438.98,124.68) -- (399,467.5) ;

\draw (235.82,25.82) node [anchor=north west][inner sep=0.75pt]  [color={rgb, 255:red, 208; green, 2; blue, 27 }  ,opacity=1 ]  {$\Sigma _{1} = \Sigma$};
\draw  [color={rgb, 255:red, 255; green, 255; blue, 255 }  ,draw opacity=1 ][fill={rgb, 255:red, 255; green, 255; blue, 255 }  ,fill opacity=1 ]  (104.23,362.48) -- (147.23,362.48) -- (147.23,416.48) -- (104.23,416.48) -- cycle  ;
\draw (115.23,366.88) node [anchor=north west][inner sep=0.75pt]   {$\Omega $};
\draw (40.38,90.9) node [anchor=north west][inner sep=0.75pt]    {$\partial \Omega $};
\draw (305,105.44) node [anchor=north west][inner sep=0.75pt]    {$\textcolor[rgb]{0.81,0.75,0}{D}\textcolor[rgb]{0.81,0.75,0}{_{1}}$};
\draw (305,164.11) node [anchor=north west][inner sep=0.75pt]    {$\textcolor[rgb]{0.81,0.75,0}{D}\textcolor[rgb]{0.81,0.75,0}{_{2}}$};
\draw (305,210.7) node [anchor=north west][inner sep=0.75pt]    {$\textcolor[rgb]{0.81,0.75,0}{D}\textcolor[rgb]{0.81,0.75,0}{_{3}}$};
\draw (335,257.99) node [anchor=north west][inner sep=0.75pt]  [rotate=-88.83]  {$\textcolor[rgb]{0.75,0.7,0.03}{.....}$};
\draw (276.03,325.34) node [anchor=north west][inner sep=0.75pt]  [xslant=-0.02]  {$\Sigma _{K}$};
\draw (300,395.18) node [anchor=north west][inner sep=0.75pt]    {$\textcolor[rgb]{0.81,0.75,0}{D}\textcolor[rgb]{0.81,0.75,0}{_{K}}$};
\draw (515,84.4) node [anchor=north west][inner sep=0.75pt]  {$\Delta =\partial \Omega \backslash \overline{\Sigma }$};
\end{tikzpicture}
\caption{Schematic figure representing a chain of domains belonging to the partition of $\Omega$, $\{D_j\}_{j=1}^K$, connecting $D_1$ with $D_{K}$.\textcolor{white}{mmmmmmmmmmmmmmmmmmmmmmmmmmmm }}
\end{figure}
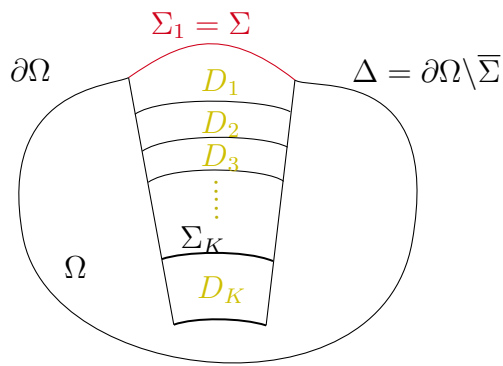
We now state our main result.
\begin{theorem}\label{Main theorem of uniqueness}
    Let $\sigma^{(i)}$ and $q^{(i)}$, for $i=1,2$, be two conductivities and potentials satisfying
     \begin{equation}\label{sigma, q result 1}
        \sigma^{(i)}(x) = \sum^N_{j=1} \sigma^{(i)}_j \chi_{D_j} (x), \qquad q^{(i)}(x) = \sum^N_{j=1} q^{(i)}_j \chi_{D_j} (x), \qquad x\in \Omega,
    \end{equation}
   respectively, where $\{D_j\}_{j=1}^N$ satisfies assumptions $(i)-(iii)$, $\sigma^{(i)}_j \in Sym_n$ is a constant matrix satisfying \eqref{ellipticity condition 0}, $q^{(i)}_j\geq 0$ is constant, for $j=1,\dots , N$ and there is $J_i\in\{1,\dots N\}$ such that $q^{(i)}_{J_i}>0$, for $i=1,2$. If 
    \[ \mathcal{N}^\Sigma_{\sigma^{(1)}, q^{(1)}} = \mathcal{N}^\Sigma_{\sigma^{(2)} , q^{(2)}},\]
    then
    \[ \sigma^{(1)} = \sigma^{(2)} \quad \textnormal{and} \quad q^{(1)}= q^{(2)}, \quad \text{in\quad $\Omega$.}\]
\end{theorem}
 Next, we consider the special case when the partition $\{D_j\}^N_{j=1}$, $j=1, \dots N$ of $\Omega$ is given by layers in $\Omega$ arising as set differences in a family of nested domains $\{\Omega_j\}_{j=0}^{N-1}$ invading $\Omega$. The precise formulation is given below (see also Figure 2).


\subsection{Unique determination of $\sigma$, $q$ and the layers' interfaces}
We assume that for some $N\in\mathbb{N}\setminus\{0\}$, there are subdomains of $\Omega$, $\Omega_1\dots , \Omega_N$, such that 
\begin{enumerate}
  \item$\;$ $\Omega_{N-1} \subset \subset \Omega_{N-2} \subset \subset\dots\Omega_1\subset \subset \Omega_{0}= \Omega,$  and we define the \textit{layers} to be the sets 
    \begin{equation}
        D_j =  \Omega_{j-1}\backslash\overline{\Omega}_{j},\quad\textnormal{for}\quad j =1,\dots,N-1\qquad\text{and}\qquad D_{N} = \Omega_{N-1}.
    \end{equation}
\item$\;$  $D_j$ is connected and $\partial\Omega_j$ is of class $C^{1, \alpha}$ and non-flat, for $j=1,\dots,N$. 
    \end{enumerate}

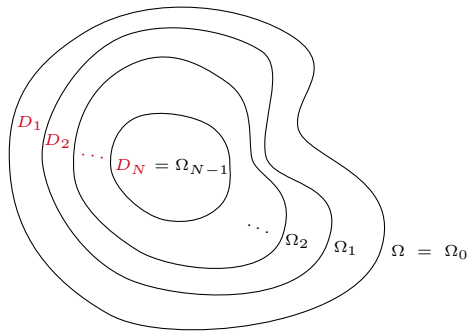
\begin{figure}[!h]
\centering
\begin{tikzpicture}[x=0.55pt,y=0.55pt,yscale=-.5,xscale=.5]

\draw   (444,51.5) .. controls (506,94.5) and (440,133.5) .. (449,175.5) .. controls (458,217.5) and (596,239.5) .. (561,344.5) .. controls (526,449.5) and (272,476.5) .. (191,430.5) .. controls (110,384.5) and (64,321.5) .. (57,239.5) .. controls (50,157.5) and (70,97.5) .. (161,42.5) .. controls (252,-12.5) and (382,8.5) .. (444,51.5) -- cycle ;
\draw   (400,70.5) .. controls (434,94.5) and (396,171.5) .. (405,213.5) .. controls (414,255.5) and (523,245.5) .. (488,335.5) .. controls (453,425.5) and (284,410.5) .. (220,383.5) .. controls (156,356.5) and (108,281.5) .. (101,223.5) .. controls (94,165.5) and (150,70.5) .. (216,47.5) .. controls (282,24.5) and (366,46.5) .. (400,70.5) -- cycle ;
\draw   (352,112.5) .. controls (391.44,140.38) and (380.28,189.27) .. (386,216.5) .. controls (391.72,243.73) and (450.26,254.43) .. (428,322.5) .. controls (405.74,390.57) and (271,368.5) .. (231,343.5) .. controls (191,318.5) and (148.19,281.33) .. (144,239.5) .. controls (139.81,197.67) and (150.12,131.16) .. (208,95.5) .. controls (265.88,59.84) and (312.56,84.62) .. (352,112.5) -- cycle ;
\draw   (331,173.5) .. controls (352.63,189.06) and (356.16,211.82) .. (356,222.5) .. controls (355.84,233.18) and (361,261.5) .. (344,280.5) .. controls (327,299.5) and (299,307.5) .. (268,301.5) .. controls (237,295.5) and (197,258.5) .. (194,229.5) .. controls (191,200.5) and (204,169.5) .. (232,159.5) .. controls (260,149.5) and (309.37,157.94) .. (331,173.5) -- cycle ;

\draw (572,332.4) node [anchor=north west][inner sep=0.75pt]    [font=\tiny]{$\Omega \ =\ \Omega _{0}$};
\draw (494,326.4) node [anchor=north west][inner sep=0.75pt]    [font=\tiny]{$\Omega _{1}$};
\draw (428,317.9) node [anchor=north west][inner sep=0.75pt]    [font=\tiny]{$\Omega _{2}$};
\draw (377.11,300.85) node [anchor=north west][inner sep=0.75pt]  [font=\tiny, rotate=-21.47]  {$\dots$};
\draw (150,205) node [anchor=north west][inner sep=0.75pt]  [font=\tiny, rotate=-10.47]  {$\textcolor[rgb]{0.82,0.01,0.11}{\dots}$};
\draw (196,215) node [anchor=north west][inner sep=0.75pt]    [font=\tiny]{{$\textcolor[rgb]{0.82,0.01,0.11}{D_{N}}$} {$=\Omega_{N-1}$}};
\draw (100,180) node [anchor=north west][inner sep=0.75pt]    [font=\tiny]{$\textcolor[rgb]{0.82,0.01,0.11}{D}\textcolor[rgb]{0.82,0.01,0.11}{_{2}}$};
\draw (62,155) node [anchor=north west][inner sep=0.75pt]    [font=\tiny]{$\textcolor[rgb]{0.82,0.01,0.11}{D}\textcolor[rgb]{0.82,0.01,0.11}{_{1}}$};
\end{tikzpicture}
\label{circle1}
\caption{The family of nested domains $\{\Omega\}_{j=0}^{N-1}$ and the layers  $\{ D\}_{j=1}^{N}$.}
\end{figure}

As a consequence of Theorem \ref{Main theorem of uniqueness}, we have the following result.


\begin{corollary}\label{Nested domain theorem}
    Let $N_i \in \mathbb{N}\backslash \{0\}$ and assume that  $\{ \Omega_j^{(i)}\}_{j=0}^{N_i-1}$,  $\{D_j^{(i)}\}_{j=1}^{N_i}$ satisfy assumptions $(i)-(ii)$, for $i=1,2$. If $\sigma^{(i)}, q^{(i)}$, for $i=1,2$ are given by
    \begin{equation}
        \sigma^{(i)}(x) = \sum^{N_i}_{j=1} \sigma^{(i)}_j \chi_{D^{(i)}_j} (x),\qquad q^{(i)}(x) = \sum^{N_i}_{j=1}  q^{(i)}_j \chi_{D^{(i)}_j} (x), \qquad x\in \Omega,
    \end{equation}
    where $\sigma^{(i)}_j \in Sym_n$ are positive definite matrices satisfying \eqref{ellipticity condition 0}, $q^{(i)}_j\geq 0$ are constant and $q_{J_i}^{(i)}>0$, for some $J_i=\{1, \dots, N_i\}$. We assume that for $i=1,2$, the following jump condition is satisfied 
    \begin{equation}
    \sigma^{(i)}_j \neq \sigma^{(i)}_{j-1}, \quad or\quad  q^{(i)}_j \neq q^{(i)}_{j-1},\quad j=1,\dots,N_i. \label{jump condition 1}
   \end{equation}
    If 
        \[ \mathcal{N}^\Sigma_{\sigma^{(1)} , q^{(1)}} = \mathcal{N}^\Sigma_{\sigma^{(2)} , q^{(2)}},\]
    then
    \begin{equation}
        N_1 = N_2 := N,
    \end{equation}
    \begin{equation}\label{Domains coincide}
        \Omega_j^{(1)} = \Omega_j^{(2)}, \qquad  \sigma_{j+1}^{(1)} = \sigma_{j+1}^{(2)},\qquad q_{j+1}^{(1)}= q_{j+1}^{(2)}, \qquad \text{for $j=0,\dots,N-1$.}
    \end{equation}
    \end{corollary}
    
The proof of Corollary \ref{Nested domain theorem}, which relies on the arguments treated in \cite{alessadrini2018eit} in the case $q=0$, can be found in the Appendix for the sake of completeness.

\section{Proof of Main Result (Theorem \ref{Main theorem of uniqueness})}
We start with the proof of Theorem \ref{Main theorem of uniqueness}, which relies on  the determination of $\sigma \in C^\alpha(\mathcal{U})$, where $\mathcal{U}$ has been introduced in \eqref{neighbourhood}, from the knowledge of $\mathcal{N}^{\Sigma}_{\sigma , q}$. This in turn allows us to determine first $\sigma$ in \eqref{sigma, q result 1} in $D_1$, then $q$ in \eqref{sigma, q result 1} on $D_1$. An inductive argument that relies on the unique continuation of ah-hoc singular solutions of \eqref{eq1} with $\sigma, q$ as in \eqref{sigma, q result 1} within $\Omega$ allows then for the unique determination of both $\sigma$ and $q$ within each $D_j$ and hence, in $\Omega$.


\subsection{Unique determination of $\sigma$ in $\mathcal{U}$}\label{sec3}

We recall that when $\sigma \in L^\infty(\Omega)$ is a matrix-valued function satisfying the uniform ellipticity condition \eqref{ellipticity condition 0} and $q\in L^\infty (\Omega)$, $q \geq 0$,
we have the equality 
\[{\frac{1}{\sqrt{\det g}}}L=-\Delta_g + {\frac{1}{\sqrt{\det g}}}q,\]
where $L$ has been introduced in \eqref{L}, $g$ is given by
\begin{equation}
    g=(\det \sigma)^{\frac{1}{n-2}}\sigma^{-1}
\end{equation}
and $\Delta_g$ is the Laplace-Beltrami operator for the Riemannian manifold $\{\Omega,g\}$, see 
\cite{berger1971spectre}, \cite{mitrea2000potential} and \cite{uhlmann2009electrical}. If $q>0$ on a subset of $\Omega$ of positive Lebesgue measure, the Neumann kernel $N_{\sigma, q}^{\Omega}$ for $L$ in \eqref{L} and $\Omega$ is, for any $y \in \Omega$, the distributional solution to   
\begin{equation}\label{Neumannprob}
 \Bigg\{\begin{array}{llll}
       LN^{\Omega}_{\sigma, q}(\cdot, y)=\delta(\cdot -{y}), \quad &\text{in} \quad  \Omega,\\
        \sigma \nabla N^{\Omega}_{\sigma, q}(\cdot, y)\cdot \nu = 0, \quad &\text{on} \quad \partial\Omega,
        \end{array}
        \end{equation}
     which is uniquely determined and satisfies, 
        \begin{equation}
            N_{\sigma, q}^{\Omega}(x,y) = N_{\sigma, q}^{\Omega}(y,x), \qquad \text{for all $x,y\in\Omega,$ $\quad x\neq y$.}
        \end{equation}
(see \cite{kim2024neumann}). $N_{\sigma, q}^{\Omega}(x,y)$ extends continuously up to the boundary $\partial\Omega$ (for $x\neq y$), so that for $y\in\partial\Omega$, it solves
        \begin{equation}
 \Bigg\{\begin{array}{llll}
       LN^{\Omega}_{\sigma, q}(\cdot, y)=0, \quad &\text{in} \quad  \Omega,\\
        \sigma \nabla N^{\Omega}_{\sigma, q}(\cdot, y)\cdot \nu = -\delta(\cdot -{y}), \quad &\text{on} \quad \partial\Omega.
        \end{array}
        \end{equation}
As in \cite{Al2017}, under some mild local smoothness assumptions on $\sigma$ in $\mathcal{U}$, we have  the following asymptotic behaviour of $N_{\sigma, q}^{\Omega}$ near its pole $y\in\Sigma$.


    \begin{theorem}\label{Neumann kernel theorem}
        Let $y\in\partial\Omega$, $\mathcal{U}$ and $\Sigma$ be as in \eqref{neighbourhood} and assume that $\sigma$, $q$ in \eqref{L} are such that $\sigma \in C^{\alpha}(\mathcal{U}\cap \overline{\Omega})$ and $q>0$ on a subset of $\Omega$ of positive Lebesgue measure. Then we have
        \begin{equation}\label{Neumann kernel asympt}
            N_{\sigma, q}^{\Omega}(x, y) =2C_n(\det(\sigma (y)))^{-\frac{1}{2}}\Big(\sigma^{-1}(y)(x-y)\cdot(x-y)\Big)^{\frac{2-n}{2}}+ O(|x-y|^{2-n+\alpha}),
        \end{equation}
        as $x\to y$, $x\in \overline{\Omega}\backslash \{y\}$ and $C_n = \frac{1}{n(n-2)\omega_n}$, where $\omega_n$ denotes the volume of the unit ball in $\mathbb{R}^n$. 
    \end{theorem}


\begin{proof}[Proof of Theorem \ref{Neumann kernel theorem}]
  $\;$ We adapt the arguments in \cite{Al2017}, \cite{mitrea2000potential}. Let $r>0$ be such that $\overline{B}_r \subset \mathcal{U}$. For any $\psi\in C^{0,1}_0(B_r)$, we have 
    \begin{equation}\label{integral on Br}
        \int_{\Omega \cap B_r} \sigma \nabla N_{\sigma, q}^{\Omega}(x, 0) \cdot \nabla \psi (x) \:\textnormal{d}x + \int_{\Omega \cap B_r} \psi (x) q(x) N_{\sigma, q}^{\Omega}(x, 0)  \:\textnormal{d}x = -\psi (0) .
    \end{equation}
Introducing the change of coordinates 
\begin{equation}
     \Bigg\{\begin{array}{llll}
         z' &=x', \\
         z_n &= x_n - \varphi (x'),
    \end{array}
\end{equation}
where $\varphi$ has been introduced in \eqref{phi 1}-\eqref{phi 2}, and setting $J= {\frac{\partial z}{\partial x}}$, we have 
\begin{equation}
    z= x + O(|x'|^{1+\alpha}),\qquad J = I + O(|x'|^{\alpha}).
\end{equation}
Defining
\begin{equation}
    \tilde{\sigma}(z) = \Big({\frac{1}{\det(J)}}J\sigma J^T \Big) (x(z))\quad\textnormal{and}\quad\tilde{N}(z)= N_{\sigma, q}^{\Omega}(x(z), 0),
\end{equation}
\eqref{integral on Br} in the new coordinates becomes
    \begin{equation}
        \int_{\{z_n>0\}} \tilde{\sigma}(z) \nabla_z \tilde{N}(z) \cdot \nabla_z \psi (x(z)) \:\textnormal{d}z +  \int_{\{z_n>0\}} \psi (z) q(z) \tilde{N}(z)\:\textnormal{d}z =- \psi (0),
    \end{equation} 
where 
\begin{equation}\label{defsigma}
    \tilde{\sigma}(z) = \sigma(0) + O(|z'|^{\alpha}).
\end{equation}
We introduce the operator
\begin{equation}\label{L bar}
L_0=-\mbox{div}(\sigma(0) \nabla \cdot),\qquad\textnormal{on}\quad \mathbb{R}^{n}_+ =\{x\in\mathbb{R}^n\quad\vert\quad x\geq 0\},
\end{equation}
and denote $\Pi_n = \{x=(x',x_n)\in \mathbb{R}^n\quad\vert\quad x_n=0\}$. For any ${y'}\in\Pi_n$, the Neumann kernel, $N_0$ for $L_0$ in \eqref{L bar} is the distributional solution to  
    \begin{equation}
      \Bigg\{\begin{array}{llll}
      L_0 N_{0}(\cdot, {y'})=0, \quad &\text{in} \quad  \mathbb{R}^{n}_+,\\
        \sigma \nabla N_{0}(\cdot, y')\cdot \nu = -\delta(\cdot -{y'}),\quad &\text{on} \quad \Pi_n,\\
        N_{0}(x,{y'}) \to 0 \quad &\text{as} \quad |x| \to \infty.
        \end{array}
        \end{equation}
For every $x\in \mathbb{R}^{n}_+$ and ${y'}\in \Pi_n$ we have
        \begin{equation}
            N_{0}(x, {y'}) = 2C_n(g_0(x-{y'})\cdot (x-{y'}))^{\frac{2-n}{2}},
        \end{equation}
        where $C_n$ is the constant introduced above (see \cite[Lemma 3.2]{Al2017}).
Thus, setting
\begin{equation}
    R(z) = \tilde{N}(z) - N_0(z,0), 
\end{equation}
we have
\begin{equation}
    \begin{split}
        \int_{\{z_n>0\}} \sigma (0)\nabla_z R(z)\cdot \nabla_z \psi (x(z)) \:\textnormal{d}z  
 =& \int_{\{z_n>0\}} ( \sigma (0) -  \tilde{\sigma}(z)) \nabla_z \tilde{N}(z) \cdot \nabla_z \psi (x(z)) \:\textnormal{d}z\\
        &-\int_{\{z_n>0\}} \psi (z) q(z) \tilde{N}(z)\:\textnormal{d}z.\\
    \end{split}
\end{equation}
    Hence for a sufficiently small $\rho>0$, we have
    \begin{equation}\label{newsystem}
    \Bigg\{ \begin{array}{llll}
           -\mbox{div}_z (\sigma (0)\nabla_z R(z)) = -\mbox{div}_z(( \sigma(0) -  \tilde{\sigma}(z)) \nabla_z \tilde{N}(z) - q(z) \tilde{N}(z), \quad &\text{in $B^{+}_\rho$} \\
         \sigma(0)\nabla_z R(z) \cdot \nu = (( \sigma(0) -  \tilde{\sigma}(z)) \nabla_z \tilde{N}(z) )\cdot \nu, 
         \quad &\text{on $B_\rho \cap \Pi_n$.}
    \end{array}
    \end{equation}
    Recalling that 
    \begin{equation}
        |N^{\Omega}_{\sigma, q}(x,0)| \leq C|x|^{2-n}, \quad \text{for every $x \in \Omega$,}
    \end{equation}
 where $C>0$ depends only on $n,\lambda, ||q||_{L^{\frac{n}{2}}(\Omega)}$ and the Lipschitz character of $\Omega$ (see \cite[Section 5]{kim2024neumann}). By the local regularity of $\sigma$, $q$ and $\Sigma$ \cite[Chapter 6]{gilbarg1977elliptic}, we also obtain
    \begin{equation}\label{Ngradestimate}
        |\nabla_x N^{\Omega}_{\sigma, q}(x,0)|  \leq C|x|^{1-n}, \quad \text{for every $x \in B_\rho \cap \Omega$.}
    \end{equation}
    Therefore, 
    \begin{equation}
        |R(z)| + |z| |\nabla_z R(z)| \leq C,  \quad \text{for every $x \in \partial B_\rho \cap \mathbb{R}^{n}_+$.}
    \end{equation}
   Hence, for every $w \in B^{+}_\rho$, we obtain
\begin{equation}
\begin{split}
    R(w) =& -\int_{\partial B^{+}_\rho}\Big( R(z) \sigma (0) \nabla_z N_0 (z,w) \cdot \nu - N_0 (z,w)\sigma (0) \nabla_z R(z)\cdot \nu\Big) \:\textnormal{d}S(z) \\
    &+\int_{B^{+}_\rho} \big(  \sigma (0) -  \tilde{\sigma}(z)\big) \nabla_z N_0(z,w) \cdot \nabla_z \tilde{N}(z) \:\textnormal{d}z -\int_{B^{+}_\rho} N_0(z,w) q(z) \tilde{N}(z)\:\textnormal{d}z\\
    &-\int_{\partial B^{+}_\rho} N_0(z,w) \big(  \sigma (0) -  \tilde{\sigma}(z)\big) \nabla_z \tilde{N}(z) \cdot \nu \:\textnormal{d}S(x).
    \end{split}
\end{equation}
Splitting $\partial B^{+}_{\rho} = ( \partial B_{\rho} \cap \mathbb{R}^{n}_+) \cup (B_\rho \cap \Pi_n),$ we get
\begin{equation}\label{defR}
\begin{split}
    R(w) =& -\int_{\partial B_\rho \cap \mathbb{R}^{n}_+} \Big(R(z) \sigma (0) \nabla_z N_0 (z,w) \cdot \nu - N_0 (z,w)\sigma (0) \nabla_z R(z)\cdot \nu \Big)\:\textnormal{d}S(z) \\
    &-\int_{\partial B_\rho \cap \mathbb{R}^n_+} N_0(z,w) \big( \sigma (0) -  \tilde{\sigma}(z)\big) \nabla_z \tilde{N}(z) \cdot \nu \:\textnormal{d}S(x)\\
     &+\int_{B^+_\rho} \big( \sigma (0) -  \tilde{\sigma}(z)\big) \nabla_z N_0(z,w) \cdot \nabla_z \tilde{N}(z) \:\textnormal{d}z -\int_{B^{+}_\rho} N_0(z,w) q(z) \tilde{N}(z)\:\textnormal{d}z. \\
    \end{split}
\end{equation}
Taking $|w|< {\frac{\rho}{2}}$, all the boundary integrals in \eqref{defR} are uniformly bounded and by \eqref{defsigma} and \eqref{Ngradestimate}, the volume integral of \eqref{defR} can be estimated as follows (see \cite[Chapter 2]{miranda2013partial})
\begin{equation}
\begin{split}
    &\Bigg| \int_{B^+_\rho }\big( \sigma (0) -  \tilde{\sigma}(z)\big) \nabla_z N_0(z,w) \cdot \nabla_z \tilde{N}(z) \:\textnormal{d}z +\int_{B^{+}_\rho} N_0(z,w) q(z) \tilde{N}(z) \:\textnormal{d}z \Bigg| 
\\
&\leq C \int_{B^+_\rho }|z'|^\alpha |z-w|^{1-n} |z|^{1-n} \:\textnormal{d}z \leq C |w|^{2-n+\alpha},
    \end{split}
\end{equation}
therefore, $|R(z)|\leq C|z|^{2-n+\alpha}$ on $B^+_\rho$ and recalling that $|z| = O(|x|)$, \eqref{Neumann kernel asympt} is proven. 
\end{proof}

\begin{lemma}\label{lemma 3.5}
    
   We assume that $y, \:\mathcal{U},\: \Sigma,\: \sigma$ and $q$ satisfy the hypotheses of Theorem \ref{Neumann kernel theorem}. Then the knowledge of $N_{\sigma,q}^\Omega(x,{y})$, for every $x\in \partial \Omega \cap \mathcal{U}$, uniquely determines
    \[ g_{(n-1)}({y}) = \{g({y})\upsilon_i \cdot \upsilon_j\}_{i,j=1,\dots,(n-1)}, \]
    where $\upsilon_1,\dots,\upsilon_{n-1}$ is a basis for the tangent space $T_{{y}}(\partial\Omega)$. 
\end{lemma}
The proof of this lemma was provided in \cite[Lemma 3.6]{Al2017} for the case $q=0$. As for the case $q\neq 0$ treated here, the proof is a straightforward adaptation of that in \cite{Al2017}, we simply refer the reader to the latter for a detailed proof.

\begin{lemma}\label{determination of sigma near boundary}
    We assume that $y, \:\mathcal{U},\: \Sigma,\: \sigma$ and $q$ satisfy the hypotheses of Theorem \ref{Neumann kernel theorem}. Moreover, we assume that $\sigma \in L^\infty(\Omega, Sym_n)$ is constant on $\mathcal{U}$, and $q>0$ on a subset of $\Omega$ of positive Lebesgue measure, then the knowledge of $N_{\sigma, q}^\Omega(x,{y})$ for every $x,{y} \in \Sigma$ uniquely determines $\sigma$ on $\mathcal{U}$.
\end{lemma}


\begin{proof}[Proof of Lemma \ref{determination of sigma near boundary}]

\: The proof follows closely the line of reasoning of \cite[Proposition 5.1]{alessandrini2024determining}, where $g$ was quantitatively estimated when $q=0$. Here our approach is qualitative as we aim at identifying $g$ when $q\neq 0$, hence we include this proof for the sake of completeness as it is a slight variation to that in \cite[Proposition 5.1]{alessandrini2024determining}. We denote the canonical basis in $\mathbb{R}^n$ by $\{ e_1, \dots , e_n\}$ and assume, without loss of generality, that $y_1 =0$, that the tangent space to $\Sigma$ at the point $0$ is $T_0 ( \partial \Sigma) = \langle e_1, \dots , e_{n-1} \rangle $ and the outer unit normal at $0$ is $\nu (y_1) = -e_n$.
Without loss of generality, we assume that there exist $\gamma_1, \gamma_2 , \gamma_3 \in [ 0,1)$, such that 
\[\nu(y_2) = \frac{-e_n + \gamma_1 e_{n-1}}{\sqrt{1 +\gamma_1^2}}\quad\textnormal{and}\quad \nu(y_3) = \frac{-e_n + \gamma_2 e_{n-1} + \gamma_3 e_{n-2}}{\sqrt{1 + \gamma_2^2 + \gamma_3^2}}.\]
Denoting by $\Pi_1 = T_0(\partial\Sigma)$, $\Pi_2, \Pi_3$ the tangent spaces at $y_2, y_3$ on $\Sigma$, respectively, we observe that $\Pi_1, \Pi_2, \Pi_3$ are generated by 
\begin{align} \label{orthonormal basis}
    & \Pi_1 = \langle e_1, \dots, e_{n-1} \rangle,\\
    &\Pi_2 = \Bigg\langle e_1, \dots , e_{n-2}, \frac{e_{n-1} + \gamma_1 e_n}{\sqrt{1 + \gamma^2_1}} \Bigg \rangle,\\
    &\Pi_3 = \langle e_1, \dots, e_{n-3}, e_{n-2} + \gamma_3e_n, e_{n-1} + \gamma_2 e_n \rangle,\label{orthonormal basis 3}
\end{align}
respectively. The non-flat condition \eqref{nonflatness} on $\Sigma$ implies that 
\[ \nu(0) \cdot \nu(y_2) = \frac{1}{\sqrt{1 +\gamma_1^2}} < 1,\]
leading to 
\[ \gamma_1 \neq 0. \]
We also have that 
\[ \nu(0) \cdot \nu (y_3) = \frac{1}{\sqrt{1 + \gamma_2^2 + \gamma_3^2}}  < 1. \]
Hence we have that 
\[ \gamma_2^2 + \gamma_3^2 >0.\]
In other words, we have  that
\[ \gamma_2 \neq 0, \quad \textnormal{or} \quad  \gamma_3 \neq 0.\]
We also have 
\[ \nu (y_2)\cdot \nu (y_3) = \frac{1 + \gamma_1 \gamma_2}{\sqrt{1 + \gamma_1^2}  \sqrt{1 + \gamma_2^2+ \gamma_3^2}}  < 1.\]
Thus,
\begin{align*}
    ( 1 + \gamma_1\gamma_2)^2 &<(1 + \gamma_1^2) ( 1 + \gamma_2^2 + \gamma_3^2)\\
    & < ( \gamma_1 - \gamma_2)^2 +1 + 2 \gamma_1 \gamma_2 + \gamma_1^2 \gamma_2^2 + \gamma_1^2 \gamma_3^2 + \gamma_3^2.
\end{align*}
Therefore, 
\[   \gamma_3^2 ( 1 + \gamma_1^2) + (\gamma_1 - \gamma_2)^2>0. \]
If $\gamma_3 =0$, we we have $\gamma_1 \neq \gamma_2$. To summarise, $\gamma_1,\: \gamma_2, \:\gamma_3$ must satisfy
       \begin{equation}\label{condition gamas}
  \{\gamma_1 \neq 0\quad\textnormal{and}\quad \gamma_3 \neq 0\}\quad\textnormal{or}\quad\{\gamma_1 \neq 0, \quad \gamma_3 =0 \quad \mbox{and} \quad  \gamma_2\neq 0, \quad \mbox{with} \quad \gamma_1 \neq \gamma_2\}.
        \end{equation}
From Lemma \ref{lemma 3.5}, we can recover the tangential component of $g$ over the first tangential plane $\Pi_1$
\begin{equation}\label{g upper left matrix}
   g_{ij}= g e_i \cdot e_j, \quad \text{for\quad $i,\:j=1, \dots , n-1.$}
\end{equation}
From the tangential component of $g$ over $\Pi_2$, we have that the following quantities
\begin{equation}\label{g over Pi_2}
    ge_i \cdot \frac{e_{n-1}+ \gamma_1 e_n}{\sqrt{1 + \gamma^2_1}} = \frac{1}{\sqrt{1 + \gamma^2_1}} (g_{n-1, i} + \gamma_1 g_{n,i}), \quad \text{for $i=1, \dots , n-2$}
\end{equation}
are known. From \eqref{condition gamas}-\eqref{g over Pi_2} we also have that
\begin{equation} \label{g_n,i}
    g_{n,i}, \quad \text{for any $i = 1, \dots ,n-2$}
\end{equation}
are known. To determine the remaining quantities, $g_{n-1, n}, g_{n,,n}$, we consider the tangential component of $g$ over $\Pi_2$, which is also known
\begin{equation}
    g \Bigg( \frac{e_{n-1}+ \gamma_1 e_n}{\sqrt{1 + \gamma^2_1}} \Bigg) \cdot \Bigg( \frac{e_{n-1}+ \gamma_1 e_n}{\sqrt{1 + \gamma^2_1}} \Bigg) = \frac{1}{\sqrt{1 + \gamma^2_1}} (g_{n-1, n-1} + 2 \gamma_1 g_{n-1,n} + \gamma_1^2 g_{n,n}).
\end{equation}
Since, by \eqref{g upper left matrix}, $g_{n-1, n-1}$ is known, the quantity 
\begin{equation}\label{Determination of last entries of g 1}
    2 g_{n-1,n} + \gamma_1 g_{n,n} 
\end{equation}
 is known too.
From the known tangential component of $g$ over $\Pi_3$, the quantity
\begin{equation}
    g \Bigg( \frac{e_{n-2}+ \gamma_3 e_n}{\sqrt{1 + \gamma^2_3}} \Bigg) \cdot \Bigg( \frac{e_{n-2}+ \gamma_3 e_n}{\sqrt{1 + \gamma^2_3}} \Bigg) = \frac{1}{\sqrt{1 + \gamma^2_3}} (g_{n-2, n-2} + 2 \gamma_3 g_{n-2,n} + \gamma_3^2 g_{n,n}),
\end{equation}
is also known, hence, by \eqref{g upper left matrix} and \eqref{g_n,i}, we have that 
\begin{equation}
    \gamma^2_3 g_{n,n} 
\end{equation}
is known. Finally considering again the tangential component of $g$ over $\Pi_3$,
\begin{equation}
      g \Bigg( \frac{e_{n-1}+ \gamma_2 e_n}{\sqrt{1 + \gamma^2_2}} \Bigg) \cdot \Bigg( \frac{e_{n-1}+ \gamma_2 e_n}{\sqrt{1 + \gamma^2_2}} \Bigg) = \frac{1}{\sqrt{1 + \gamma^2_2}} (g_{n-1, n-1} + 2 \gamma_2 g_{n-1,n} + \gamma_2^2 g_{n,n})
\end{equation}
is known. 
By \eqref{g upper left matrix},  
\begin{equation}\label{Determination of last entries of g 2}
    2 \gamma_2 g_{n-1,n} + \gamma_2^2 g_{n,n}
\end{equation}
is known. Gathering together the terms in \eqref{Determination of last entries of g 1}-\eqref{Determination of last entries of g 2}, we have  that $\mathcal{AG}\in \mathbb{R}^3$ is known with
\[ \mathcal{A} = \begin{pmatrix}
2 & \gamma_1 \\
0 & \gamma_3^2\\
2\gamma_2 & \gamma_2^2
\end{pmatrix},	\quad  G = \begin{pmatrix}
g_{n-1,n}\\
g_{n,n}
\end{pmatrix}. 	\]
If the first condition in \eqref{condition gamas} holds, then $G\in \mathbb{R}^2$ is determined by inverting $\mathcal{A}_1 = \begin{pmatrix}
2 & \gamma_1 \\
0 & \gamma_3^2
\end{pmatrix},$
where, if the second condition in \eqref{condition gamas} holds, then $G$ is determined by inverting $\mathcal{A}_2 = \begin{pmatrix}
2 & \gamma_1 \\
2\gamma_2 & \gamma_2^2
\end{pmatrix}.$
\end{proof}
\begin{remark}\label{remark K}
The knowledge of the full N-D map is equivalent to knowing the boundary values of the Neumann kernel (see \cite[Remark 3.7]{Al2017}). However, only the asymptotic behaviour of $N_{\sigma, q}(x,y)$, for $x,y \in \Sigma$ and $x \rightarrow y$ can be determined by the local N-D map $\mathcal{N}^\Sigma_{\sigma, q}$, when $q=0$. Hence, the following adjustment to $N_{\sigma, 0}(x,y)$ was made in \cite{Al2017},
     \begin{equation}\label{K for q0}
     K_{\sigma, 0} (x,y,w,z) = N_{\sigma, 0}(x,y)-N_{\sigma, 0}(x,w) -N_{\sigma, 0}(z,y)+N_{\sigma, 0}(z,w) 
      \end{equation}
for $x,y,w,z \in \Sigma$ distinct points and it was shown that knowing $\mathcal{N}^\Sigma_{\sigma, 0}$ is equivalent to knowing $K_{\sigma, 0}(x,y,w,z)$ for any $x,y,w,z \in \Sigma$, which asymptotic behavior, for fixed $w,z \in\Sigma$, is the same as $N_{\sigma, 0}(x,y)$, as $x \rightarrow y$. 
\end{remark}
As a similar argument to that in \cite{Al2017} can be adopted here to the case when the non-negative $q$ is positive on a set of positive Lebesgue measure, we define
\begin{equation}\label{K for q}
     K_{\sigma, q} (x,y,w,z) = N_{\sigma, q}(x,y)-N_{\sigma, q}(x,w) -N_{\sigma, q}(z,y)+N_{\sigma, q}(z,w) 
      \end{equation}
for $x,y,w,z \in \Sigma$ distinct points, and observe that remark \ref{remark K} is valid also in the current setting and refer to \cite{Al2017} for its proof.

\subsection{Unique determination of $\sigma$ and $q$ on $D_1$}\label{sec4}
Next, we show that, once we have determined $\sigma$ on $\Sigma$, taking advantage of the \textit{a-priori} assumptions \eqref{sigma, q result 1}, \eqref{ellipticity condition 0} on $\sigma$ and $q \in L^\infty (\Omega)$, a non-negative scalar-valued function in $\Omega$ such that $q >0$ on a subset of $\Omega$ of positive Lebesgue measure, we can uniquely determine both $\sigma$ and $q$ in $D_1$.
\begin{proposition}\label{V on D1}
Let $\Omega$, $\Sigma$, ${y'}_0\in\Sigma$, as in Lemma \ref{determination of sigma near boundary}. Assume that $\sigma^{(i)}, q^{(i)}\in L^{\infty}(\Omega)$ are as in \eqref{sigma, q result 1}, $\sigma^{(i)}$ satisfies \eqref{ellipticity condition 0} and $q_{J_i}^{(i)}>0$ for some $J_i \in \{1, \dots,N\}$, for $i=1,2$. If 
\begin{equation}\label{equality maps}
\mathcal{N}^{\Sigma}_{\sigma^{(1)}, q^{(1)}} = \mathcal{N}^{\Sigma}_{\sigma^{(2)}, q^{(2)}},
\end{equation}
then
\begin{equation}\label{sigma V on D1}
\sigma^{(1)} = \sigma^{(2)}\qquad\textnormal{and}\qquad q^{(1)} = q^{(2)},\qquad\textnormal{in}\quad D_1.
\end{equation}
\end{proposition}
\begin{proof}[Proof of Proposition \ref{V on D1}]
$\;$ As a straightforward consequence of Lemma \ref{determination of sigma near boundary}, we have
\begin{equation}\label{identification sigma}
\sigma^{(1)} = \sigma^{(2)},\qquad\textnormal{in}\quad D_1.
\end{equation}
Let $B_{\frac{r_0}{4}}(P_1)$ be such that $B_{\frac{r_0}{4}}(P_1) \subset \subset \mathcal{U}$, for $r_0$ as per definition \ref{LipschitzBoundary}. For $y,z \in B_{\frac{r_0}{4}}(P_1)\backslash\overline{\Omega},$ we have

\begin{eqnarray}\label{Singular solutions 2}
 & & \Big\langle \sigma^{(1)} \nabla \tilde{N}^{(1)}(y,\cdot)\cdot \nu, \Big(\mathcal{N}^{\Sigma}_{\sigma^{(2)},q^{(2)}} - \mathcal{N}^{\Sigma}_{\sigma^{(1)}, q^{(1)}}\Big) \sigma^{(2)} \nabla \tilde{N}^{(2)}(y,\cdot)\cdot \nu \Big\rangle\nonumber\\
 & & =   \int_{\Omega} \Big(\sigma^{(1)}-\sigma^{(2)}\Big)(x)\nabla_x \tilde{N}^{(1)}(y,x) \cdot \nabla_x \tilde{N}^{(2)}(z,x)\:\textnormal{d}x\\
 && +\int_{\Omega}\Big(q^{(1)}- q^{(2)}\Big)(x)\tilde{N}^{(1)}(y,x) \cdot \tilde{N}^{(2)}(z,x)\:\textnormal{d}x \nonumber,
      \end{eqnarray}
By combining \eqref{equality maps} together with \eqref{Singular solutions 2}, we obtain
\begin{equation}
\begin{split}
0 &=\int_{B_{\frac{r_0}{4}}(P_1)\cap D_1}\Big(q^{(1)}- q^{(2)}\Big)(x)\tilde{N}^{(1)}(y,x) \:\tilde{N}^{(2)}(z,x)\:\textnormal{d}x\\
&+\int_{\Omega\setminus (B_{\frac{r_0}{4}}(P_1)\cap D_1)} \Big(\sigma^{(1)}-\sigma^{(2)}\Big)(x)\nabla_x \tilde{N}^{(1)}(y,x)\cdot \nabla_x \tilde{N}^{(2)}(z,x)\:\textnormal{d}x\\
&+\int_{\Omega\setminus (B_{\frac{r_0}{4}}(P_1)\cap D_1)}\Big(q^{(1)}- q^{(2)}\Big)(x)\tilde{N}^{(1)}(y,x) \:\tilde{N}^{(2)}(z,x)\:\textnormal{d}x, 
\end{split}
\end{equation}
and
\begin{equation}\label{sing sol 3}
\begin{split}
0 &=\int_{B_{\frac{r_0}{4}}(P_1)\cap D_1}\Big(q^{(1)}- q^{(2)}\Big)(x)\partial_{y_n}\tilde{N}^{(1)}(y,x) \:\partial_{z_n}\tilde{N}^{(2)}(z,x)\:\textnormal{d}x\\
&+\int_{\Omega\setminus (B_{\frac{r_0}{4}}(P_1)\cap D_1)} \Big(\sigma^{(1)}-\sigma^{(2)}\Big)(x)\nabla_x \partial_{y_n}\tilde{N}^{(1)}(y,x)\cdot \nabla_x \partial_{z_n}\tilde{N}^{(2)}(z,x)\:\textnormal{d}x\\
& +\int_{\Omega\setminus (B_{\frac{r_0}{4}}(P_1)\cap D_1)}\Big(q^{(1)}- q^{(2)}\Big)(x)\partial_{y_n}\tilde{N}^{(1)}(y,x) \:\partial_{z_n}\tilde{N}^{(2)}(z,x)\:\textnormal{d}x.
\end{split}
\end{equation}
Defining $y_1 = P_1+r\nu(P_1)$, with $0<r<\frac{r_0}{8}$, where $\nu$ denotes the outer unit normal to $\partial{D}_1$, and setting $y=z=y_1$ in \eqref{sing sol 3}, we obtain
\begin{equation}\label{sing sol 4}
\begin{split}
0 &= \int_{B_{\frac{r_0}{4}}(P_1)\cap D_1}\Big(q^{(1)}- q^{(2)}\Big)(x)\partial_{y_n}\tilde{N}^{(1)}(y_1,x) \:\partial_{z_n}\tilde{N}^{(2)}(y_1,x)\:\textnormal{d}x\\
&+\int_{\Omega\setminus (B_{\frac{r_0}{4}}(P_1)\cap D_1)} \Big(\sigma^{(1)}-\sigma^{(2)}\Big)(x)\nabla_x \partial_{y_n}\tilde{N}^{(1)}(y_1,x)\cdot \nabla_x \partial_{z_n}\tilde{N}^{(2)}(y_1,x)\:\textnormal{d}x\\
& +\int_{\Omega\setminus (B_{\frac{r_0}{4}}(P_1)\cap D_1)}\Big(q^{(1)}- q^{(2)}\Big)(x)\partial_{y_n}\tilde{N}^{(1)}(y_1,x) \:\partial_{z_n}\tilde{N}^{(2)}(y_1,x)\:\textnormal{d}x.
\end{split}
\end{equation}
Defining
\begin{equation}\label{Gammai}
\Gamma_i(x,y):=C_n\left(\det(\sigma^{(i)}(y))\right)^{-\frac{1}{2}}\left((\sigma^{(i)}(y))^{-1}(x-y)\cdot (x-y)\right)^{\frac{2-n}{n}},
\end{equation}
where  $C_n>0$ is the constant introduced in \eqref{Neumann kernel asympt}, due to the regularity of $\sigma^{(i)}$, $q^{(i)}$ on $\overline{D}_1$, for $i=1,2$ (see \cite{mitrea2000potential}), and the Caccioppoli inequatlity, \eqref{sing sol 4} leads to
\begin{equation}
    \begin{split}
       & |q_1^{(1)} - q_1^{(2)}|\:\int_{B_{\frac{r_0}{4}}(P_1)\cap D_1} |\partial_{y_n}\Gamma_1(x,y_1)|\:|\partial_{z_n}\Gamma_2(x,y_1)|\:\textnormal{d}x\\ 
       &\leq C\bigg\{\int_{B_{\frac{r_0}{4}}(P_1)\cap D_1} |x-y_1|^{2-2n+\alpha} \:\textnormal{d}x+\int_{B_{\frac{r_0}{4}}(P_1)\cap D_1} |x-y_1|^{2-2n+2\alpha} \:\textnormal{d}x \\
       & +  r_0^{2-n} + r_0^{-n}\bigg\}, 
       \end{split}
\end{equation}
hence
\begin{equation}
   |q_1^{(1)} - q_1^{(2)}| \:\int_{B_{\frac{r_0}{4}}(P_1)\cap D_1} |x-y_1|^{2-2n}\:\textnormal{d}x\\ \leq C\lambda (r_0^{2-n} + r_0^{-n} + r^{2-n+\alpha} + r^{2-n + 2\alpha}),
\end{equation}
so that
\begin{equation}\label{sing sol 5}
    \begin{split}
        |q_1^{(1)} - q_1^{(2)}| &\leq C(r^{n-2} + r^{\alpha} + r^{2\alpha}).
    \end{split}
\end{equation}
By letting $r \rightarrow 0$ in \eqref{sing sol 5}, we obtain 
\begin{equation}
        q^{(1)}=q^{(2)}, \qquad \text{in\quad $D_1$},
\end{equation}
which concludes the proof of \eqref{sigma V on D1}.
\end{proof}


\subsection{Unique determination of $\sigma$ and $q$ in $\Omega$}\label{sec5}

\begin{proof}[Proof of Theorem \ref{Main theorem of uniqueness}]
$\;$ Without loss of generality we assume $\Sigma = \Sigma_1$. If $\mathcal{N}^{\Sigma_1}_{\sigma^{(1)}, q^{(1)}} = \mathcal{N}^{\Sigma_1}_{\sigma^{(2)}, q^{(2)}}$, then by Lemmas \ref{determination of sigma near boundary}, Proposition \ref{V on D1}, we have
\begin{equation}
\sigma^{(1)} = \sigma^{(2)} \quad\text{and}\quad q^{(1)} = q^{(2)} \quad\text{in $D_1$}.
\end{equation}
Let $D_K$ be a subdomain of $\Omega$, introduced in section \ref{sec general partition}, with $1 \neq k\leq N$. There exists $j_1,\dots, j_K \in \{ 1,\dots N\}$ such that 
\begin{equation}\label{Division of D}
D_{j_1} = D_1, \dots, D_{j_K} = D_K. 
\end{equation}
For simplicity, we denote the chain of subdomains mentioned above by $D_1, \dots, D_K$, $K \leq N$. 
We proceed by induction and assume that $\sigma^{(1)} = \sigma^{(2)}$ and $q^{(1)} = q^{(2)}$, in $D_i$, for every $1 \leq i \leq K$ and show that $\sigma^{(1)} = \sigma^{(2)}$ and $q^{(1)} = q^{(2)}$ in $D_{K+1}$ too.
 We find it convenient to consider a larger domain that contains $\Omega$. To this purpose, note that up to a rigid transformation of coordinates we can assume that 
    \[P_1 =0; \quad (\mathbb{R}^n\backslash\Omega) \cap B_{r_0} = \{(x',x_n) \in B_{r_0}|\quad x_n < \phi(x')\},\]
    where $\phi$ is the Lipschitz function introduced definition \ref{LipschitzBoundary}.
Define
\[ D_0 = \left\{x \in (\mathbb{R}^n\backslash\Omega )\cap B_{r_0}\Big|\quad |x_i| <{\frac{2}{3}}r_0,\quad i =1,\dots,n-1,\quad |x_n - \frac{r_0}{6}|<\frac{5}{6}r_0\right\}.\]
The augmented domain $\Omega_0 = \Omega \cup D_0$ is of Lipschitz class with Lipschitz character depending on that of $\partial\Omega$. For any number $r \in (0, \frac{2}{3}r_0)$, we set
\[ (D_0)_r = \{x \in D_0 \;|\;  \mbox{dist}(x, \Omega) >r\}.\]
We denote by $L_i$ the operator introduced in \eqref{L} with $\sigma = \sigma^{(i)}$ and $q=q^{(i)}$, for $i=1,2$. We extend $\sigma^{(i)}$ and $q^{(i)}$ to $\tilde{\sigma}^{(i)}$ and $\tilde{q}^{(i)}$, respectively, on $\Omega_0$, by setting $\tilde{\sigma}^{(i)} |_{D_0}=\sigma^{(i)}_1$ and $\tilde{q}^{(i)}|_{D_0}=q^{(i)}_1$ and continue to denote the extended operators on $\Omega_0$ by $L_i$, 
\begin{equation}\label{Li extended}
    L_i = 
    -\mbox{div}(\tilde{\sigma}^{(i)}\nabla \cdot) +\tilde{q}^{(i)}\cdot,\quad \text{in}\quad\Omega_0,
\end{equation}
for $i=1,2$. For $y\in \Omega_0$, we define the Neumann kernel $\tilde{N}_{\tilde{\sigma}^{(i)}, \tilde{q}^{(i)}}$ as the solution to

 \begin{equation}\label{Neumann kernel externded}
\left\{ \begin{array}{lll}\displaystyle L_i \tilde{N}^{\Omega}_{\tilde{\sigma}^{(i)},\tilde{q}^{(i)}}(\cdot, y)=\delta(\cdot -{y}), \quad &\text{in} \quad  \Omega_0,\\
 \tilde{\sigma}^{(i)} \nabla \tilde{N}^{\Omega}_{\tilde{\sigma}^{(i)}, \tilde{q}^{(i)}}(\cdot, y)\cdot \nu = 0, \quad &\text{on} \quad \partial\Omega_0,\\
\end{array} \right.
\end{equation}   
    (see, for example, \cite{kim2024neumann}).
Notice that 
\begin{equation}\label{Neumann kernel symmetry}
\tilde{N}^{\Omega}_{\tilde{\sigma}^{(i)}, {\tilde{q}}^{(i)}}(x,y) = \tilde{N}^{\Omega}_{\tilde{\sigma}^{(i)}, {\tilde{q}}^{(i)}}(y,x), \quad \forall x,y \in \Omega_0, \quad x \neq y.
\end{equation}
We also simplify the notation by setting
        \[\tilde{N}^{\Omega}_{\tilde{\sigma}^{(i)}, \tilde{q}^{(i)}} = \tilde{N}^{(i)}, \quad \text{for $i =1,2$,}\]
and define
\begin{equation}\label{partioned domain}
D = \Bigg(\bigcup_{i=1}^N \overline{D_i} \Bigg); \quad E = \Omega \backslash \overline{D}.
\end{equation}
We introduce $\mathcal{N}^{\Sigma_{K+1}}_{\sigma^{(i)}, q^{(i)}}$, the local D-N map for the domain $E$ relative to $\sigma^{(i)}$, $q^{(i)}$ and localised on $\Sigma_{K+1}$, for $i =1,2.$

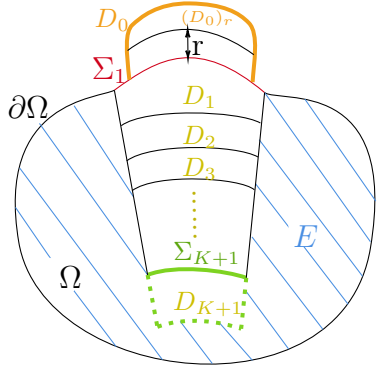
\begin{figure}[!h]
\centering

\begin{tikzpicture}[x=0.75pt,y=0.75pt,yscale=-.3,xscale=.3]

\draw [color={rgb, 255:red, 0; green, 0; blue, 0 }  ,draw opacity=1 ]   (446.5,194.5) .. controls (471.5,212.5) and (670.5,165.5) .. (622.5,455.5) .. controls (574.5,745.5) and (77.5,684.5) .. (37.5,445.5) .. controls (-2.5,206.5) and (171.5,204.5) .. (196.5,190.5) ;
\draw    (208,250) .. controls (248,220) and (402.5,224.5) .. (439.5,245.5) ;
\draw    (217,309) .. controls (257,279) and (398.5,281.5) .. (435.5,302.5) ;
\draw    (226,359) .. controls (266,329) and (406,337) .. (429.5,357.5) ;
\draw [color={rgb, 255:red, 126; green, 211; blue, 33 }  ,draw opacity=1 ][line width=1.5]    (251.83,500.17) .. controls (298.67,482) and (382,493) .. (416,505) ;
\draw  [color={rgb, 255:red, 255; green, 255; blue, 255 }  ,draw opacity=1 ][fill={rgb, 255:red, 255; green, 255; blue, 255 }  ,fill opacity=1 ] (379.73,518.91) .. controls (365.62,517.55) and (342.52,516.67) .. (316.44,516.67) .. controls (298.46,516.67) and (281.92,517.09) .. (268.79,517.79) -- (316.64,522) -- cycle ;
\draw [color={rgb, 255:red, 208; green, 2; blue, 27 }  ,draw opacity=1 ]   (196.5,190.5) .. controls (271.5,135.5) and (335.5,100.5) .. (446.5,194.5) ;
\draw [color={rgb, 255:red, 74; green, 144; blue, 226 }  ,draw opacity=1 ]   (32,387) -- (85.5,465.5) -- (124.5,520.5) -- (199.5,620.5) ;
\draw [color={rgb, 255:red, 74; green, 144; blue, 226 }  ,draw opacity=1 ]   (39,321) -- (276.5,642.5) ;
\draw [color={rgb, 255:red, 74; green, 144; blue, 226 }  ,draw opacity=1 ]   (71.83,246.17) -- (251.83,500.17) ;
\draw [color={rgb, 255:red, 74; green, 144; blue, 226 }  ,draw opacity=1 ]   (129.5,216.5) -- (228.5,369.5) ;
\draw [color={rgb, 255:red, 74; green, 144; blue, 226 }  ,draw opacity=1 ]   (306,578) -- (359.5,648.83) ;
\draw [color={rgb, 255:red, 74; green, 144; blue, 226 }  ,draw opacity=1 ]   (372,581.33) -- (426.17,643.5) ;
\draw [color={rgb, 255:red, 74; green, 144; blue, 226 }  ,draw opacity=1 ]   (415.5,527.5) -- (486.5,624.5) ;
\draw [color={rgb, 255:red, 74; green, 144; blue, 226 }  ,draw opacity=1 ]   (425.5,436.5) -- (543.5,591.5) ;
\draw [color={rgb, 255:red, 74; green, 144; blue, 226 }  ,draw opacity=1 ]   (431.5,347.5) -- (585.5,549.5) ;
\draw [color={rgb, 255:red, 74; green, 144; blue, 226 }  ,draw opacity=1 ]   (437.5,273.5) -- (609.5,502.5) ;
\draw [color={rgb, 255:red, 74; green, 144; blue, 226 }  ,draw opacity=1 ]   (446.5,194.5) -- (622.5,444.5) ;
\draw [color={rgb, 255:red, 74; green, 144; blue, 226 }  ,draw opacity=1 ]   (513.5,205.5) -- (631.5,376.5) ;
\draw [color={rgb, 255:red, 241; green, 159; blue, 26 }  ,draw opacity=1 ][line width=1.5]    (221.5,173) .. controls (216.5,127) and (213.5,93) .. (226.5,81) .. controls (239.5,69) and (267.5,44) .. (322.5,45) .. controls (377.5,46) and (403.5,60) .. (420.5,78) .. controls (437.5,96) and (428.5,145) .. (426.5,177) ;
\draw [color={rgb, 255:red, 126; green, 211; blue, 33 }  ,draw opacity=1 ][line width=1.5]  [dash pattern={on 1.69pt off 2.76pt}]  (251.83,500.17) -- (268.12,587.38) ;
\draw [color={rgb, 255:red, 126; green, 211; blue, 33 }  ,draw opacity=1 ][line width=1.5]  [dash pattern={on 1.69pt off 2.76pt}]  (416,505) -- (407.5,590.5) ;
\draw [color={rgb, 255:red, 126; green, 211; blue, 33 }  ,draw opacity=1 ][line width=1.5]  [dash pattern={on 1.69pt off 2.76pt}]  (268.12,587.38) .. controls (314.95,569.21) and (373.5,578.5) .. (407.5,590.5) ;
\draw    (218.5,124) .. controls (258.5,94) and (330.5,59) .. (428.5,131) ;
\draw    (318.52,93) -- (318.98,133) ;
\draw [shift={(319,135)}, rotate = 269.35] [color={rgb, 255:red, 0; green, 0; blue, 0 }  ][line width=0.75]    (10.93,-3.29) .. controls (6.95,-1.4) and (3.31,-0.3) .. (0,0) .. controls (3.31,0.3) and (6.95,1.4) .. (10.93,3.29)   ;
\draw [shift={(318.5,91)}, rotate = 89.35] [color={rgb, 255:red, 0; green, 0; blue, 0 }  ][line width=0.75]    (10.93,-3.29) .. controls (6.95,-1.4) and (3.31,-0.3) .. (0,0) .. controls (3.31,0.3) and (6.95,1.4) .. (10.93,3.29)   ;
\draw    (196.5,190.5) -- (251.83,500.17) ;
\draw    (446.5,194.5) -- (416,505) ;

\draw (155,130.4) node [anchor=north west][inner sep=0.75pt]  [color={rgb, 255:red, 208; green, 2; blue, 27 }  ,opacity=1 ]  {$\Sigma _{1}$};
\draw  [color={rgb, 255:red, 255; green, 255; blue, 255 }  ,draw opacity=1 ][fill={rgb, 255:red, 255; green, 255; blue, 255 }  ,fill opacity=1 ]  (85.5,472.5) -- (128.5,472.5) -- (128.5,526.5) -- (85.5,526.5) -- cycle  ;
\draw (100.5,480.9) node [anchor=north west][inner sep=0.75pt]  {$\Omega $};
\draw (15,195.4) node [anchor=north west][inner sep=0.75pt]  {$\partial \Omega $};
\draw (305,179.4) node [anchor=north west][inner sep=0.75pt]    [font=\footnotesize]{$\textcolor[rgb]{0.81,0.75,0}{D}\textcolor[rgb]{0.81,0.75,0}{_{1}}$};
\draw (305,245.73) node [anchor=north west][inner sep=0.75pt]    [font=\footnotesize]{$\textcolor[rgb]{0.81,0.75,0}{D}\textcolor[rgb]{0.81,0.75,0}{_{2}}$};
\draw (305,300) node [anchor=north west][inner sep=0.75pt]    [font=\footnotesize]{$\textcolor[rgb]{0.81,0.75,0}{D}\textcolor[rgb]{0.81,0.75,0}{_{3}}$};
\draw (335.78,351.35) node [anchor=north west][inner sep=0.75pt]  [rotate=-88.83]  {$\textcolor[rgb]{0.75,0.7,0.03}{......}$};
\draw  [color={rgb, 255:red, 255; green, 255; blue, 255 }  ,draw opacity=1 ][fill={rgb, 255:red, 255; green, 255; blue, 255 }  ,fill opacity=1 ]  (487,413) -- (513,413) -- (513,452) -- (487,452) -- cycle  ;
\draw (490,412.4) node [anchor=north west][inner sep=0.75pt]   {$\textcolor[rgb]{0.29,0.56,0.89}{E}$};
\draw (160,50) node [anchor=north west][inner sep=0.75pt]    [font=\footnotesize]{$\textcolor[rgb]{0.95,0.62,0.1}{D}\textcolor[rgb]{0.95,0.62,0.1}{_{0}}$};
\draw (290.07,440.4) node [anchor=north west][inner sep=0.75pt]  [font=\footnotesize][xslant=-0.02]  {$\textcolor[rgb]{0.36,0.65,0.05}{\Sigma }\textcolor[rgb]{0.36,0.65,0.05}{_{K+1}}$};
\draw (290,525.19) node [anchor=north west][inner sep=0.75pt]    [font=\footnotesize]{$\textcolor[rgb]{0.81,0.75,0}{D}\textcolor[rgb]{0.81,0.75,0}{_{K+1}}$};
\draw (303,52.4) node [anchor=north west][inner sep=0.75pt]    [font=\tiny]{$\textcolor[rgb]{0.95,0.62,0.1}{(}\textcolor[rgb]{0.95,0.62,0.1}{D}\textcolor[rgb]{0.95,0.62,0.1}{_{0}}\textcolor[rgb]{0.95,0.62,0.1}{)}\textcolor[rgb]{0.95,0.62,0.1}{_{r}}$};
\draw (321,104) node [anchor=north west][inner sep=0.75pt]   [align=left] {r};
\end{tikzpicture}

\caption{Schematic figure displaying the non-physical domain $D_0$, subdomain $(D_0)_r$ and the augmented domain $\Omega_0$.}
\end{figure}

\begin{claim}\label{claim}
    If $\mathcal{N}^{\Sigma_{1}}_{\sigma^{(1)}, q^{(1)}} = \mathcal{N}^{\Sigma_{1}}_{\sigma^{(2)}, q^{(2)}}$, $\sigma^{(1)} = \sigma^{(2)}$ and $q^{(1)} = q^{(2)}$ in $D$, then  
    \begin{equation}\label{claim 6}
    \mathcal{N}^{\Sigma_{K+1}}_{\sigma^{(1)}, q^{(1)}} =  \mathcal{N}^{\Sigma_{K+1}}_{\sigma^{(2)}, q^{(2)}}.
    \end{equation}
\end{claim}

\begin{proof}[Proof of Claim \ref{claim}]
   $\;$ Here we use a similar argument of \cite{Al2017}, \cite{alessandrini2012single}. 
        Given $\psi \in C^{0,1}(\partial E)$, with $\mbox{supp} \psi \subset \Sigma_{K+1}$, let $u^{(i)}$ solve 
         \begin{equation}
 \Bigg\{\begin{array}{llll}
       L_i u^{(i)}=0, \quad &\text{in} \quad  E,\\
       \sigma^{(i)} \nabla u^{(i)}\cdot \nu = \psi, \quad &\text{on} \quad \partial E.
        \end{array}
        \end{equation}
As in \cite{Al2017}, we consider a bounded extension operator
\[T: H^{\frac{1}{2}} (\partial E \cap \Omega) \longrightarrow H^1 (\Omega),\]
such that, given $f \in H^{\frac{1}{2}}(\partial E \cap \Omega)$, we have 
\[T f|_{\Sigma_1} = 0,\]
and define
\begin{equation}
    \overline{u}^{(i)} = \Bigg\{\begin{array}{llll}
    u^{(i)}, \quad \text{in} \quad E,\\
    T(u^{(i)}|_{\partial E \cap \Omega}), \quad \text{in} \quad D.
      \end{array}
\end{equation}
Observe that $\overline{u}^{(i)}\in H^1(\Omega)$, hence for $x \in E$, we obtain
\begin{equation}\label{u representation}
     \begin{split}
         \overline{u}^{(i)}(x) =&- \int_\Omega \overline{u}^{(i)}(y) \mbox{div}_y(\sigma^{(i)} (y) \nabla_y \tilde{N}^{(i)}(y,x) \:\textnormal{d}y + \int_\Omega \overline{u}^{(i)}(y) q^{(i)}(y) \tilde{N}^{(i)}(y,x) \:\textnormal{d}y\\
           =&  \int_{\Sigma_{K+1}}  \psi \tilde{N}^{(i)}(y,x)  \:\textnormal{d}S(y)+ \int_D \sigma^{(i)} (y) \nabla_y \overline{u}^{(i)}(y) \cdot \nabla_y \tilde{N}^{(i)}(y,x) \:\textnormal{d}y \\
          &+\int_D \overline{u}^{(i)}(y) q^{(i)} (y)\tilde{N}^{(i)}(y,x) \:\textnormal{d}y.
     \end{split}
 \end{equation}
Introducing the quantities
 \[F_{\sigma}(x,y,z) = \big(\sigma^{(1)} - \sigma^{(2)}\big)(x)\nabla_x \tilde{N}^{(1)}(y,x) \cdot \nabla_x \tilde{N}^{(2)}(z,x), \]
 \[ F_{q}(x,y,z) = \big(q^{(1)} - q^{(2)}\big)(x)\tilde{N}^{(1)}(y,x) \tilde{N}^{(2)}(z,x),\]
 and differentiating under the integrals in \eqref{u representation}, we form for $x\in E$, 
 \begin{equation*}
     \begin{split}
        & \big(\sigma^{(1)} - \sigma^{(2)}\big)(x)\nabla_x  u^{(1)}(x) \cdot \nabla_x  u^{(2)}(x)\\
        = &  \int_{\Sigma_{K+1} \times \Sigma_{K+1}} \!\!\!\!\!\!\!\!\!\!\!\!\!\!\!\!\!\!\!\!\!\! \psi(y) \psi (z) F_{\sigma}(x,y,z)  \:\textnormal{d}y \:\textnormal{d}z\\
+ & \int_{\Sigma_{K+1} \times D} \!\!\!\!\!\!\!\!\!\!\!\!\!\!\!\! \psi(y) \big[ \sigma^{(2)}_{lk} (z) \partial_{z_l} \overline{u}^{(2)}(z) \partial_{z_k} F_{\sigma}(x,y,z) +q^{(2)} (z) \overline{u}^{(2)}(z) F_{\sigma}(x,y,z) \big] \:\textnormal{d}y \:\textnormal{d}z\\
+ & \int_{D \times \Sigma_{K+1}}\!\!\!\!\!\!\!\!\!\!\!\!\!\!\!\!  \psi(z) \big[ \sigma^{(1)}_{lk} (y) \partial_{y_l} \overline{u}^{(1)}(y) \partial_{y_k} F_{\sigma}(x,y,z) + q^{(1)} (y) \overline{u}^{(1)}(y)  F_{\sigma}(x,y,z) \big] \:\textnormal{d}y \:\textnormal{d}z\\
 + & \int_{D \times D} \!\!\!\!\!\!\!\! \overline{u}^{(1)}(y) \overline{u}^{(2)}(z) q^{(1)} (y) q^{(2)} (z)F_{\sigma}(x,y,z)  \:\textnormal{d}y \:\textnormal{d}z\\
    + & \int_{D \times D}\!\!\!\!\!\!\!\! \sigma^{(1)}_{lk} (y) q^{(2)} (z) \partial_{y_l} \overline{u}^{(1)}(y) \overline{u}^{(2)}(z) \partial_{y_k}  F_{\sigma}(x,y,z)  \:\textnormal{d}y \:\textnormal{d}z\\
 + & \int_{D \times D} \!\!\!\!\!\!\!\!\sigma^{(2)}_{lk} (z) \overline{u}^{(1)}(y) q^{(1)} (y) \partial_{z_l} \overline{u}^{(2)}(z)  \partial_{z_k}  F_{\sigma}(x,y,z)  \:\textnormal{d}y \:\textnormal{d}z\\
 + & \int_{D \times D}\!\!\!\!\!\!\!\!\sigma^{(2)}_{lk} (z) \partial_{z_l} \overline{u}^{(2)}(z) \sigma^{(1)}_{nm} (y) \partial_{y_n} \overline{u}^{(1)}(y) \partial_{z_k} \partial_{y_l} F_{\sigma}(x,y,z)  \:\textnormal{d}y \:\textnormal{d}z,\\
     \end{split}
 \end{equation*}
 where the summation convention over repeated indices has been adopted. We also have 
 \begin{equation*}
     \begin{split}
        &\big(q^{(1)} - q^{(2)}\big)(x) u^{(1)}(x) u^{(2)}(x) \\
        =& \int_{\Sigma_{K+1} \times \Sigma_{K+1}} \!\!\!\!\!\!\!\!\!\!\!\!\!\!\!\!\!\!\!\!\!\!\!\!\psi(y) \psi (z)  F_{q}(x,y,z)  \:\textnormal{d}y \:\textnormal{d}z\\
 + &\int_{\Sigma_{K+1} \times D} \!\!\!\!\!\!\!\!\!\!\!\!\!\!\!\!\psi(y) \big[ \sigma^{(2)}_{lk} (z) \partial_{z_l} \overline{u}^{(2)}(z) \partial_{z_k} F_{q}(x,y,z) +q^{(2)} (z) \overline{u}^{(2)}(z) F_{q}(x,y,z)\big] \:\textnormal{d}y \:\textnormal{d}z\\
 + &\int_{D \times \Sigma_{K+1}} \!\!\!\!\!\!\!\!\!\!\!\!\!\!\!\! \psi(z) \big[ \sigma^{(1)}_{lk} (y) \partial_{y_l} \overline{u}^{(1)}(y) \partial_{y_k} F_{q}(x,y,z) + q^{(1)}(y) \overline{u}^{(1)}(y)  F_{q}(x,y,z) \big]  \:\textnormal{d}y \:\textnormal{d}z\\
 + &\int_{D \times D}\!\!\!\!\!\!\!\! \overline{u}^{(1)}(y) \overline{u}^{(2)}(z) q^{(1)} (y) q^{(2)} (z)F_{q}(x,y,z) \:\textnormal{d}y \:\textnormal{d}z\\
    + &\int_{D \times D} \sigma^{(1)}_{lk} (y) q^{(2)} (z) \partial_{y_l} \overline{u}^{(1)}(y) \overline{u}^{(2)}(z) \partial_{y_k}  F_{q}(x,y,z) \:\textnormal{d}y \:\textnormal{d}z\\
 + &\int_{D \times D}\!\!\!\!\!\!\!\! \sigma^{(2)}_{lk} (z) \overline{u}^{(1)}(y) q^{(1)} (y) \partial_{z_l} \overline{u}^{(2)}(z)  \partial_{z_k}  F_{q}(x,y,z) \:\textnormal{d}y \:\textnormal{d}z\\
    + &\int_{D \times D}\!\!\!\!\!\!\!\! \sigma^{(2)}_{lk} (z) \partial_{z_l} \overline{u}^{(2)}(z) \sigma^{(1)}_{nm} (y) \partial_{y_n} \overline{u}^{(1)}(y) \partial_{z_k} \partial_{y_l} F_{q}(x,y,z) \:\textnormal{d}y \:\textnormal{d}z.\\
     \end{split}
 \end{equation*}
We define for $y,z \in D \cup D_0$, 
 \begin{eqnarray}\label{Singular solutions}
     S(y,z) &=& \int_{E} \big(\sigma^{(1)} - \sigma^{(2)}\big)(x) \nabla_x \tilde{N}^{(1)}(y,x) \cdot \nabla_x \tilde{N}^{(2)}(z,x)\: \textnormal{d}x\\
     &+& \int_{E} \big(q^{(1)}- q^{(2)}\big)(x)\tilde{N}^{(1)}(y,x) \tilde{N}^{(2)}(z,x)\: \textnormal{d}x
 \end{eqnarray}
and observe that for any $y,z \in (\Bar{D} \cup \Bar{D}_0)^\circ$, we have
\begin{equation}
    \begin{split}
       & -\mbox{div}_y\big(\sigma^{(1)}(y) \nabla_y S(y,z)\big) +q^{(1)}(y) S(y,z)= 0,\\
       & -\mbox{div}_z\big(\sigma^{(2)}(z) \nabla_z S(y,z)\big) +q^{(2)}(z) S(y,z)= 0.
    \end{split}
\end{equation}
Since $\sigma^{(1)} = \sigma^{(2)}$ and $q^{(1)} = q^{(2)}$ on $D$, by assumption, we also have
 \begin{equation}
 \begin{split}
     S(y,z) = &\int_\Omega \Big(\sigma^{(1)}-\sigma^{(2)}\Big)(x) \nabla_x \tilde{N}^{(1)}(y,x) \cdot \nabla_x \tilde{N}^{(2)}(z,x) \: \textnormal{d}x \\
     &+ \int_\Omega\Big(q^{(1)}-q^{(2)}\Big)(x)  \tilde{N}^{(1)}(y,x)  \tilde{N}^{(2)}(z,x) \: \textnormal{d}x.
      \end{split}
 \end{equation}
Thus, for $y, z \in (D_0)_r$, 
\begin{equation}\label{Alessandrini id for k-1}
    S(y,z) = \Big\langle \sigma^{(1)} \nabla \tilde{N}^{(1)}(y,\cdot)\cdot \nu, \Big(\mathcal{N}^{\Sigma_1}_{\sigma^{(2)}} - \mathcal{N}^{\Sigma_1}_{\sigma^{(1)}}\Big) \sigma^{(2)} \nabla \tilde{N}^{(2)}(y,\cdot)\cdot \nu \Big\rangle = 0.
\end{equation}
Due to the $C^{1,\alpha}$ regularity of the interfaces $\Sigma_{j_k}$ within D, $S(y,z)$ satisfies the unique continuation property in each variable $y,z\in (\overline{D} \cup \overline{D}_0)^\circ$ \cite[Theorem 1.1]{li2003estimates}, therefore
\begin{equation}
    S(y,z) =0, \qquad \text{for any $y,z \in D$.}
\end{equation}
Hence, 
\begin{equation}
    \begin{split}
       & \int_E \Big[\big(\sigma^{(1)}-\sigma^{(2)}\big)(x) \nabla_x u^{(1)}(x) \cdot \nabla_x u^{(2)}(x) + \big(q^{(1)}-q^{(2)}\big)(x)  u^{(1)}(x)  u^{(2)}(x)\Big] \:\textnormal{d}x\\
    = & \int_{\Sigma_{K+1} \times \Sigma_{K+1}} \!\!\!\!\!\!\!\! \!\!\!\!\!\!\!\!  \psi(y) \psi (z) S(y,z) \:\textnormal{d}y \: \textnormal{d}z \\
 + &\int_{D \times \Sigma_{K+1}}\!\!\!\!\!\!\!\!  \psi(z)\Big[  \sigma^{(1)}_{lk} (y) \partial_{y_l} \overline{u}^{(1)}(y) \partial_{y_k} S(y,z)  + \overline{u}^{(1)}(y) q^{(1)} (y) S(y,z)\Big]\: \textnormal{d}y \:\textnormal{d}z\\
   + &\int_{\Sigma_{K+1} \times D} \!\!\!\!\!\!\!\!  \psi(y)\Big[ q^{(2)} (z) \overline{u}^{(2)}(z) S(y,z) +  \sigma^{(2)}_{lk} (z) \partial_{z_l} \overline{u}^{(2)}(z) \partial_{z_k} S(y,z) \Big] \:\textnormal{d}y \:\textnormal{d}z\\ 
    + &\int_{D \times D}  \overline{u}^{(1)}(y) \overline{u}^{(2)}(z) q^{(1)} (y) q^{(2)} (z)S(y,z) \:\textnormal{d}y \:\textnormal{d}z\\
    + &\int_{D \times D}\sigma^{(1)}_{lk} (y) q^{(2)} (z) \partial_{y_l} \overline{u}^{(1)}(y)\overline{u}^{(2)}(z)\partial_{y_k} S(y,z)\: \textnormal{d}y \: \textnormal{d}z\\
   + &\int_{D \times D} \sigma^{(2)}_{lk} (z) q^{(1)} (y) \overline{u}^{(1)}(y)\partial_{z_l} \overline{u}^{(2)}(z)    \partial_{z_k}  S(y,z) \:\textnormal{d}y \:\textnormal{d}z\\
 + &\int_{D \times D} \sigma^{(2)}_{lk} (z) \partial_{z_l} \overline{u}^{(2)}(z) \sigma^{(1)}_{nm} (y) \partial_{y_n} \overline{u}^{(1)}(y) \partial_{z_k} \partial_{y_l} S(y,z) \:\textnormal{d}y \:\textnormal{d}z=0.
    \end{split}
\end{equation}
Finally, we obtain, 
\begin{equation}
\begin{split}
& \Big\langle \psi, \Big(\mathcal{N}^{\Sigma_{K+1}}_{\sigma^{(2)}} - \mathcal{N}^{\Sigma_{K+1}}_{\sigma^{(1)}}\Big)\psi \Big\rangle \\
   & = \int_E \big(\sigma^{(1)} - \sigma^{(2)}\big)(x) \nabla_x u^{(1)}(x) \cdot \nabla_x u^{(2)}(x) + \big(q^{(1)} - q^{(2)}\big)(x)  u^{(1)}(x)  u^{(2)}(x) \:\textnormal{d}x = 0,
    \end{split}
\end{equation}
which concludes the proof of Claim \ref{claim}.
\end{proof}
From $\mathcal{N}^{\Sigma_{K+1}}_{\sigma^{(2)}} = \mathcal{N}^{\Sigma_{K+1}}_{\sigma^{(1)}}$ and by Lemma \ref{determination of sigma near boundary}, we obtain
\[\sigma^{(1)}_{K+1} = \sigma^{(2)}_{K+1}, \qquad \text{on\quad $ \Sigma_{K+1}$,}\]
and hence 
\[\sigma^{(1)}_{K+1} = \sigma^{(2)}_{K+1} \quad \mbox{and} \quad q^{(1)}_{K+1} = q^{(2)}_{K+1}, \qquad \text{in\quad $D_{K+1}$.}\]

\end{proof}


\appendix
\section{Appendix}\label{sec6}
We prove here Corollary \ref{Nested domain theorem}. We also state and prove a slight variation of Corollary \ref{Nested domain theorem}, where $\sigma$ and $q$ in \eqref{L} are uniquely determined only locally in $\Omega$ within a layered subregion $C$ of $\Omega$, from the relevant local N-D map. No layered assumption is made on $\Omega\setminus\overline{C}$ .

\begin{proof}[Proof of Corollary \ref{Nested domain theorem}] $\;$ 
Without loss of generality, let $N_1 = \min \{ N_1, N_2\}$. We note that the case $j=0$, $\Omega_0^{(1)} = \Omega_0^{(2)} = \Omega$ is trivially true. Let
    \[ \mathcal{N}^\Sigma_{\sigma^{(1)} , q^{(1)}} = \mathcal{N}^\Sigma_{\sigma^{(2)} , q^{(2)}},\]
   and $E_1$ be the connected component of $\Omega \backslash \overline{(\Omega_1^{(1)} \cup \Omega_1^{(2)})}$, such that $\Sigma \subset \partial E_1$. By Theorem \ref{Main theorem of uniqueness}, we obtain
    \begin{equation}
        \sigma_{1}^{(1)} = \sigma_{1}^{(2)} \quad \textnormal{and} \quad q_{1}^{(1)}= q_{1}^{(2)}, \quad \text{in $E_1$.}
    \end{equation}
We now proceed by induction for $1 \leq k \leq N_1$. Assume for every $j = 1,\dots, k-1$
\begin{equation} \label{induction assumption}
      \Omega_j^{(1)} = \Omega_j^{(2)}, \quad  \sigma_{j+1}^{(1)} = \sigma_{j+1}^{(2)} \quad \textnormal{and} \quad q_{j+1}^{(1)}= q_{j+1}^{(2)}, 
\end{equation}
and assume by contradiction that 
\begin{equation}\label{contradiction}
    \partial \Omega^{(1)}_k \backslash \overline{\Omega^{(2)}_k} \neq \emptyset.
\end{equation}
Let $E_k$ be the connected component of $\Omega \backslash \overline{(\Omega_k^{(1)} \cup \Omega_k^{(2)})}$, such that $\Sigma \subset \partial E_k$. Denote $\Sigma_{k+1}$ an open portion of $(\partial \Omega^{(1)}_k \backslash \overline{\Omega^{(2)}_k}) \cap \partial E_k$ and choose a subdomain $\mathcal{E}_k \subset E_k$ such that $\mathcal{E}_k $ and the subdomain $\mathcal{F}_k = \Omega \backslash \overline{\mathcal{E}_k}$ both have Lipschitz boundary and $\Sigma \cup \Sigma_{k+1} \subset \partial \mathcal{E}_k$ (see Fig. \ref{Graph of subdomains for nested domains} illustrating the case $k=1$). In $\mathcal{E}_k$, we have that $\sigma^{(1)} = \sigma^{(2)}$ and $q^{(1)} = q^{(2)}$. We denote the local N-D map for $\sigma^{(i)}$, $q^{(i)}$ in $\mathcal{F}_k$ by 
\[\mathcal{N}^{\Sigma_{k+1}}_{\sigma^{(i)} , q^{(i)}}.\] 
Then by Claim \ref{claim}, we have that 
\begin{equation}\label{N-D maps coinciding}
    \mathcal{N}^{\Sigma_{k+1}}_{\sigma^{(1)} , q^{(1)}} = \mathcal{N}^{\Sigma_{k+1}}_{\sigma^{(2)} , q^{(2)}},
\end{equation}
where we have replaced set $D$ in Claim \ref{claim} with $\mathcal{E}_k$. \\


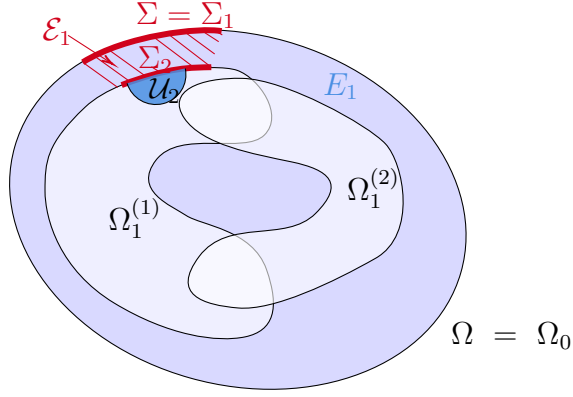
\begin{figure}[!h] 
\centering
\begin{tikzpicture}[x=0.75pt,y=0.75pt,yscale=-.5,xscale=.5]

\draw  [fill=blue!50,  fill opacity=.3]  (109.86,143.08) .. controls (136.49,49.6) and (257.75,2.21) .. (380.72,37.23) .. controls (503.68,72.26) and (581.77,176.44) .. (555.14,269.92) .. controls (528.51,363.4) and (407.25,410.79) .. (284.28,375.77) .. controls (161.32,340.74) and (83.23,236.56) .. (109.86,143.08) -- cycle ;
\draw  [fill=white!50, fill opacity=.6]   (332,66.5) .. controls (354,74.5) and (379,111.5) .. (359,131.5) .. controls (339,151.5) and (275,125.5) .. (251,156.5) .. controls (227,187.5) and (266,202.5) .. (281,212.5) .. controls (296,222.5) and (347,226.5) .. (361,266.5) .. controls (375,306.5) and (368,328.5) .. (340,334.5) .. controls (312,340.5) and (263,333.5) .. (209,300.5) .. controls (155,267.5) and (140,216.5) .. (140,187.5) .. controls (140,158.5) and (141,140.5) .. (180,103.5) .. controls (219,66.5) and (310,58.5) .. (332,66.5) -- cycle ;
\draw  [fill=white!50, fill opacity=.6]  (337,74.5) .. controls (380,75.5) and (462,121) .. (480,140.5) .. controls (498,160) and (512,231.5) .. (470,255.5) .. controls (428,279.5) and (368,308.5) .. (333,305.5) .. controls (298,302.5) and (265,292.5) .. (294,253.5) .. controls (323,214.5) and (380,240.5) .. (410,210.5) .. controls (440,180.5) and (424,166.5) .. (386,158.5) .. controls (348,150.5) and (297,141.5) .. (279,114.5) .. controls (261,87.5) and (294,73.5) .. (337,74.5) -- cycle;
\draw [color={rgb, 255:red, 208; green, 2; blue, 27 }  ,draw opacity=1 ][line width=2.25]    (177.67,58.33) .. controls (221,31) and (294.33,23) .. (316.33,27);
\draw [color={rgb, 255:red, 208; green, 2; blue, 27 }  ,draw opacity=1 ][line width=2.25]    (216.5,81.25) .. controls (264.5,62.75) and (286.5,63.75) .. (306,63.25) ;
\draw [color={rgb, 255:red, 208; green, 2; blue, 27 }  ,draw opacity=1 ]   (177.67,58.33) -- (212.6,83.6) ;
\draw [color={rgb, 255:red, 208; green, 2; blue, 27 }  ,draw opacity=1 ]   (280.07,28.33) -- (309.4,50) ;
\draw [color={rgb, 255:red, 208; green, 2; blue, 27 }  ,draw opacity=1 ]   (262.47,29.53) -- (306,63.25) ;
\draw [color={rgb, 255:red, 208; green, 2; blue, 27 }  ,draw opacity=1 ]   (247.67,32.33) -- (287.8,62.8) ;
\draw [color={rgb, 255:red, 208; green, 2; blue, 27 }  ,draw opacity=1 ]   (204.47,46.33) -- (239.4,71.6) ;
\draw [color={rgb, 255:red, 208; green, 2; blue, 27 }  ,draw opacity=1 ]   (189.67,52.33) -- (224.6,77.6) ;
\draw [color={rgb, 255:red, 208; green, 2; blue, 27 }  ,draw opacity=1 ]   (246.2,62.8) -- (256.2,69.2) ;
\draw [color={rgb, 255:red, 208; green, 2; blue, 27 }  ,draw opacity=1 ]   (294.87,25.53) -- (313.4,37.6) ;
 \draw [color={rgb, 255:red, 208; green, 2; blue, 27 }  ,draw opacity=1 ]   (310.47,27.13) -- (317,30.8) ;
\draw [color={rgb, 255:red, 208; green, 2; blue, 27 }  ,draw opacity=1 ]   (163,31.2) -- (205.72,58.91) ;
\draw [shift={(207.4,60)}, rotate = 212.97] [color={rgb, 255:red, 208; green, 2; blue, 27 }  ,draw opacity=1 ][line width=0.75]    (10.93,-3.29) .. controls (6.95,-1.4) and (3.31,-0.3) .. (0,0) .. controls (3.31,0.3) and (6.95,1.4) .. (10.93,3.29)   ;
\draw  [draw opacity=0][fill={rgb, 255:red, 74; green, 144; blue, 226 }  ,fill opacity=0.86 ] (279.99,64.48) .. controls (283.41,80.3) and (273.63,96.06) .. (257.81,99.94) .. controls (242.23,103.76) and (226.51,94.65) .. (221.94,79.46) -- (250.67,70.8) -- cycle ; \draw   (279.99,64.48) .. controls (283.41,80.3) and (273.63,96.06) .. (257.81,99.94) .. controls (242.23,103.76) and (226.51,94.65) .. (221.94,79.46) ;  
\draw (543,317.4) node [anchor=north west][inner sep=0.75pt]    {$\mathrm{\Omega \ =\ \Omega _{0}}$};
\draw (201,192.4) node [anchor=north west][inner sep=0.75pt]    {$\Omega _{1}^{( 1)}$};
\draw (439,162.4) node [anchor=north west][inner sep=0.75pt]    {$\Omega _{1}^{( 2)}$};
\draw (228,-5) node [anchor=north west][inner sep=0.75pt]  [color={rgb, 255:red, 208; green, 2; blue, 27 }  ,opacity=1 ]  {$\Sigma = \Sigma_1 $};
\draw (230,40) node [anchor=north west][inner sep=0.75pt]  [color={rgb, 255:red, 208; green, 2; blue, 27 }  ,opacity=1 ]  {$\Sigma _{2}$};
\draw (135,12) node [anchor=north west][inner sep=0.75pt]    {$\textcolor[rgb]{0.82,0.01,0.11}{\mathcal{E}_{1}}$};
\draw (415,66.4) node [anchor=north west][inner sep=0.75pt]    {$\textcolor[rgb]{0.29,0.56,0.89}{E_{1}}$};
\draw (242,72) node [anchor=north west][inner sep=0.75pt]   [font=\small]{$\mathcal{U}_{2}$};
\end{tikzpicture}
\caption{Schematic figure representing $\Omega_1^{(1)}$, $\Omega_1^{(2)}$, the connected component $E_1$ of $\Omega \backslash \overline{(\Omega_1^{(1)} \cup \Omega_1^{(2)})}$, such that $\Sigma \subset \partial E_1$, the subdomain $\mathcal{E}_1 \subset E_1$, the open portion $\Sigma_2\subset(\partial \Omega^{(1)}_1 \backslash \overline{\Omega^{(2)}_1}) \cap \partial E_1$ and the neighborhood $\mathcal{U}_2$ of $y_2\in\Sigma_2$.}\label{Graph of subdomains for nested domains}
\end{figure}
We now choose $y_{k+1} \in \Sigma_{k+1}$ and a neighbourhood of $y_{k+1}$, $\mathcal{U}_{k+1}$ in $\mathcal{F}_k$, such that $\mathcal{U}_{k+1} \cap \overline{\Omega^{(2)}_k} = \emptyset$. For $\mathcal{U}_{k+1}$ small enough, $\sigma^{(i)} , q^{(i)}$, for $i=1,2$ are constant in $\mathcal{U}_{k+1}$, thus, from Lemma \ref{determination of sigma near boundary} and Theorem \ref{Main theorem of uniqueness}, we have that 
\begin{equation}\label{sigma and q 1}
    \sigma_{k+1}^{(1)} = \sigma_{k}^{(2)}\quad\textnormal{and}\quad q_{k+1}^{(1)} = q_{k}^{(2)}
\end{equation}
and by pairing \eqref{sigma and q 1} with \eqref{induction assumption} we obtain
\begin{equation}\label{sigma and q 2}
     \sigma_{k+1}^{(1)} = \sigma_{k}^{(1)}\quad\textnormal{and}\quad q_{k+1}^{(1)} = q_{k}^{(1)},
\end{equation}
with \eqref{sigma and q 2} contradicting the \textit{jump condition} \eqref{jump condition 1}, therefore we must have 
\begin{equation}\label{equal domains}
    \Omega_k^{(1)} = \Omega_k^{(2)}.
\end{equation}
Now, by combining together \eqref{equal domains} and \eqref{N-D maps coinciding}, by Lemma \ref{determination of sigma near boundary} and Theorem \ref{Main theorem of uniqueness}, we also have that 
\begin{equation}\label{eq4.8}
     \sigma_{k+1}^{(1)} = \sigma_{k+1}^{(2)}\quad\textnormal{and}\quad q_{k+1}^{(1)} = q_{k+1}^{(2)},
\end{equation}
therefore \eqref{Domains coincide} holds for $j = 1,\dots,N_1-1$. In particular, we have that 
\begin{equation}\label{known on D^2_K_1+1}
      \Omega_{N_1 -1}^{(1)} = \Omega_{N_1-1}^{(2)}:= \Omega_{N_1-1}, \quad  \sigma_{{N_1}}^{(1)} = \sigma_{{N_1}}^{(2)} \quad \textnormal{and} \quad q_{{N_1}}^{(1)}= q_{{N_1}}^{(2)},  \quad \text{on $D^{(2)}_{N_1}$.}
\end{equation}
To show that $N_1 = N_2 := N$, we proceed again with a proof by contradiction. Assume that $N_2>N_1$ and define
\begin{equation}
    \widetilde{\Omega}^{(1)}_k = \Omega^{(2)}_k, \quad \text{$k = N_1, \dots, N_2$},
\end{equation}
\begin{equation}
    \widetilde{D}^{(1)}_k = D^{(2)}_k, \quad \text{$k = N_1, \dots, N_2$},
\end{equation}
\begin{equation}\label{sigma K_1}
    \widetilde{\sigma}^{(1)}_k = \sigma^{(1)}_{N_1}, \quad \text{$k = N_1, \dots, N_2$},
\end{equation}
\begin{equation}\label{V K_1+1}
    \widetilde{q}^{(1)}_k = q^{(1)}_{N_1}, \quad \text{$k = N_1, \dots, N_2$}.
\end{equation}
Let $\Sigma_{N_1+1}$ be a non-empty portion of $\partial \tilde{\Omega}^{(1)}_{N_1} =\partial \tilde{\Omega}^{(2)}_{N_1}$. Adapting the same arguments already utilised in this proof: \eqref{N-D maps coinciding}-\eqref{eq4.8}, Lemma \ref{determination of sigma near boundary} and Theorem \ref{Main theorem of uniqueness}, we have that 
\begin{equation}\label{claim 6 adapted}
\sigma^{(1)} = \sigma^{(2)} \quad \textnormal{and} \quad q^{(1)}= q^{(2)},  \quad \text{on $\Omega \backslash \overline{\tilde{\Omega}^{(1)}_{N_1}} =  \Omega \backslash \overline{\tilde{\Omega}^{(2)}_{N_1}}$.}
\end{equation}
Combining \eqref{claim 6 adapted} with Claim \ref{claim}, we have 
\[     \mathcal{N}^{\Sigma_{N_1+1}}_{\sigma^{(1)} , q^{(1)}} = \mathcal{N}^{\Sigma_{N_1+1}}_{\sigma^{(2)} , q^{(2)}},\]
whence
\begin{equation}\label{sigma K_1+2}
    \sigma_{{N_1}+1}^{(1)} = \sigma_{{N_1}+1}^{(2)} \quad \textnormal{and} \quad q_{{N_1}+1}^{(1)}= q_{{N_1}+1}^{(2)}.
\end{equation}
Furthermore,  from \eqref{sigma K_1}, \eqref{V K_1+1}, we have 
\[  \sigma_{{N_1}}^{(1)} = \sigma_{{N_1}+1}^{(1)} \quad \textnormal{and} \quad q_{{N_1}}^{(1)}= q_{{N_1}+1}^{(1)},\]
which, when combined with \eqref{known on D^2_K_1+1}, \eqref{sigma K_1+2}, leads to
\[  \sigma_{{N_1}}^{(2)} = \sigma_{{N_1}+1}^{(2)} \quad \textnormal{and} \quad q_{{N_1}}^{(2)}= q_{{N_1}+1}^{(2)},\]
with the latter contradicting again the \textit{jump condition} \eqref{jump condition 1}, hence $N_1 =N_2$.
\end{proof}

As in \cite{alessadrini2018eit}, Corollary \ref{Nested domain theorem} can be extended to a setting of a locally layered structure. If $\Sigma$ is the open portion of $\partial \Omega$ where measurements are taken as in Definition \ref{nonflatness} and $\varphi$ is the function in Definition \ref{function of class C^1 alpha}, we define a a subdomain of $\Omega$, 
\begin{equation}
    C = \{ x \in \mathbb{R}^n
 |\quad |x'| \leq R, \quad\varphi \leq x_n \leq M\},
 \end{equation}
for some positive constants $R, M$, and 
\begin{equation}
    \partial C \cap \partial\Omega = \{x \in \mathbb{R}^n |\quad |x'| \leq R, \quad x_n = \varphi (x')\} \supset \Sigma.
\end{equation}
Let $\varphi_1,\dots,\varphi_N$ : $B'_R \rightarrow \mathbb{R}$ be $C^{1, \alpha}$ functions satisfying definitions \ref{function of class C^1 alpha}, \ref{nonflatness} and 
\begin{equation}
    \varphi(x') \equiv \varphi_0(x') < \varphi_1(x') <\dots< \varphi_N(x')< M, \qquad \textnormal{for all $x' \in \overline{B'_R}$.}
\end{equation}
For $j = 1,\dots,N$, we denote the partitions of $C$
\begin{equation}
    D_j = \{ x \in C |\quad\varphi_{j-1}(x') < x_n < \varphi_j(x')\},
\end{equation}
  and assume that $\sigma \in L^\infty(\Omega, Sym_n)$ is a matrix valued function satisfying the uniform ellipticity condition \eqref{ellipticity condition 0}, $q\in L^\infty (\Omega)$ is a non-negative scalar-valued function in $\Omega$ such that $q>0$ on a subset of $\Omega$ of positive Lebesgue measure and that
 \begin{equation}
        \sigma(x) = \sum^N_{j=1} \sigma_j \chi_{D_j} (x), \qquad q(x) = \sum^N_{j=1} q_j \chi_{D_j} (x), \qquad x\in C.
    \end{equation}
Under the \textit{jump condition}
    \begin{equation}
         \quad \sigma_j \neq \sigma_{j+1}\quad\textnormal{or}\quad q_j \neq q_{j+1}, \quad \textnormal{for}\quad j=1,\dots,N-1,
    \end{equation}
we have the following uniqueness result localised within $C$.

\begin{corollary}\label{unique determination in q}
    $\mathcal{N}^\Sigma_{\sigma, q}$ uniquely determines $\sigma$ and $q$ within $C$. 
\end{corollary}
\begin{proof}[Proof of Corollary \ref{unique determination in q}]
   $\;$ The proof that follows the same line of reasoning of Corollary \ref{Nested domain theorem}. Let $\sigma^{(i)}, q^{(i)}$ satisfy 
     \begin{equation}
        \sigma^{(i)}(x) = \sum^{N_i}_{j=1} \sigma^{(i)}_j \chi_{D^{(i)}_j} (x), \qquad q^{(i)}(x) = \sum^{N_i}_{j=1} q^{(i)}_j \chi_{D^{(i)}_j} (x), \qquad x\in \Omega, \quad i =1,2,
    \end{equation}
    satisfying the aforementioned conditions and the partition layers $D^{(i)}_j$ are described by the functions $\varphi_j^{(i)}$, $i=1,2$. We can select inner boundary portions $\Sigma_j$ such that they are all within $C$ due to the fact that the interfaces are graphs with respect to the same reference system. We define the sets $\mathcal{E}_k$ as 
    \[\mathcal{E}_k = \{ x \in \mathbb{R}^n \: : \: |x'| < R, \: \varphi\leq x_n \leq \min\{\varphi^{(1)}_k, \varphi^{(2)}_k\}\}.\]
    Hence,
    \[ \mathcal{N}^\Sigma_{\sigma^{(1)}, q^{(1)}} = \mathcal{N}^\Sigma_{\sigma^{(2)}, q^{(2)}}, \]
    leads to 
    \[ \sigma^{(1)} = \sigma^{(2)}\quad\textnormal{and}\quad q^{(1)} = q^{(2)},\qquad\textnormal{in}\quad C.\]
   \end{proof}

\section*{Acknowledgment}

This publication has emanated from research conducted with the financial support of Taighde \'Eireann - Research Ireland under Grant number GOIPG/2021/527. The authors would like to thank the Isaac Newton Institute for Mathematical Sciences, Cambridge, for support and hospitality during the programme Rich and Nonlinear Tomography - a multidisciplinary approach, where work on this paper was undertaken. This work was supported by EPSRC grant EP/R014604/1.

\bibliography{ref}

\begin{thebibliography}{10}

\bibitem{alessandrini1990singular}
G.~Alessandrini.
\newblock Singular solutions of elliptic equations and the determination of
  conductivity by boundary measurements.
\newblock {\em Journal of Differential Equations}, 84(2):252--272, 1990.

\bibitem{alessadrini2018eit}
G.~Alessandrini, M.~{De Hoop}, R.~Gaburro, and E.~Sincich.
\newblock {EIT in a layered anisotropic medium}.
\newblock {\em Inverse Problems and Imaging}, 12(3):667--676, 2018.

\bibitem{Al2017}
G.~Alessandrini, M.~V. de~Hoop, and R.~Gaburro.
\newblock Uniqueness for the electrostatic inverse boundary value problem with
  piecewise constant anisotropic conductivities.
\newblock {\em Inverse problems}, 33(12):125013, 2017.

\bibitem{A-G2001}
G.~Alessandrini and R.~Gaburro.
\newblock Determining conductivity with special anisotropy by boundary
  measurements.
\newblock {\em SIAM J. Math. Anal.}, 33:153--171, 2001.

\bibitem{A-G2009}
G.~Alessandrini and R.~Gaburro.
\newblock {The local Calder\'on problem and the determination at the boundary
  of the conductivity}.
\newblock {\em Commin. PDE}, 34:918--936, 2009.

\bibitem{alessandrini2024determining}
G.~Alessandrini, R.~Gaburro, and E.~Sincich.
\newblock {Determining an anisotropic conductivity by boundary measurements:
  Stability at the boundary}.
\newblock {\em Journal of Differential Equations}, 382:115--140, 2024.

\bibitem{alessandrini2012single}
G.~Alessandrini and K.~Kim.
\newblock {Single-logarithmic stability for the Calder{\'o}n problem with local
  data}.
\newblock {\em Journal of Inverse and Ill-Posed Problems}, 20(4):389--400,
  2012.

\bibitem{applegate2020recent}
M.~Applegate, R.~Istfan, S.~Spink, A.~Tank, and D.~Roblyer.
\newblock Recent advances in high speed diffuse optical imaging in biomedicine.
\newblock {\em APL Photonics}, 5(4), 2020.

\bibitem{arridge1999optical}
S.~R. Arridge.
\newblock Optical tomography in medical imaging.
\newblock {\em Inverse problems}, 15(2):R41, 1999.

\bibitem{arridge1998nonuniqueness}
S.~R. Arridge and W.~R. Lionheart.
\newblock Nonuniqueness in diffusion-based optical tomography.
\newblock {\em Optics letters}, 23(11):882--884, 1998.

\bibitem{arridge2009optical}
S.~R. Arridge and J.~C. Schotland.
\newblock Optical tomography: forward and inverse problems.
\newblock {\em Inverse problems}, 25(12):123010, 2009.

\bibitem{astala2006}
K.~Astala and L.~P{\"a}iv{\"a}rinta.
\newblock {Calder{\'o}n's inverse conductivity problem in the plane}.
\newblock {\em Annals of Mathematics}, pages 265--299, 2006.

\bibitem{As-La-P2005}
K.~Astala, L.~P{\"a}iv{\"a}rinta, and M.~Lassas.
\newblock {Calder\'on inverse problem for anisotropic conductivity in the
  plane}.
\newblock {\em Commun. PDE}, 30:207--224, 2005.

\bibitem{Belishev2003}
M.~I. Belishev.
\newblock {The Calder\'on problem for 2D manifolds by the BC-method}.
\newblock {\em SIAM J. Math. Anal.}, 35:172--182, 2003.

\bibitem{berger1971spectre}
M.~Berger, P.~Gauduchon, and E.~Mazet.
\newblock {\em Le spectre d'une vari{\'e}t{\'e} riemannienne}.
\newblock Springer, 1971.

\bibitem{brown1997uniqueness}
R.~M. Brown and G.~A. Uhlmann.
\newblock Uniqueness in the inverse conductivity problem for nonsmooth
  conductivities in two dimensions.
\newblock {\em Communications in partial differential equations},
  22(5-6):1009--1027, 1997.

\bibitem{calderon2006inverse}
A.~P. Calder{\'o}n.
\newblock On an inverse boundary value problem.
\newblock {\em Seminar on Numerical Analysis and Its Applications to Continuum
  Physics}, pages 65--73, 1980.

\bibitem{caro2016global}
P.~Caro and K.~M. Rogers.
\newblock {Global uniqueness for the Calder{\'o}n problem with Lipschitz
  conductivities}.
\newblock In {\em Forum of Mathematics, Pi}, volume~4. Cambridge University
  Press, 2016.

\bibitem{Carstea2018}
C.~C\^{a}rstea, N.~Honda, and G.~Nakamura.
\newblock Uniqueness in the inverse boundary value problem for piecewise
  homogeneous anisotropic elasticity.
\newblock {\em SIAM Journal on Mathematical Analysis}, 50(3):3291--3302, 2018.

\bibitem{curran2024singular}
J.~Curran, R.~Gaburro, and C.~Nolan.
\newblock Singular solutions for complex second order elliptic equations and
  their application to time-harmonic diffuse optical tomography.
\newblock {\em Applied Mathematics Letters}, 157:109162, 2024.

\bibitem{curran2023timeharmoic}
J.~Curran, R.~Gaburro, C.~Nolan, and E.~Somersalo.
\newblock {Time-harmonic diffuse optical tomography: Hölder stability of the
  derivatives of the optical properties of a medium at the boundary}.
\newblock {\em Inverse Problems and Imaging}, 17(2):338--361, 2023.

\bibitem{foschiatti2024lipschitz}
S.~Foschiatti.
\newblock {Lipschitz stability estimate for the simultaneous recovery of two
  coefficients in the anisotropic Schr{\"o}dinger type equation via local
  Cauchy data}.
\newblock {\em Journal of Mathematical Analysis and Applications},
  531(1):127753, 2024.

\bibitem{F-G-S2021}
S.~Foschiatti, R.~Gaburro, and E.~Sincich.
\newblock {Stability for the Calder\'on problem for a class of anisotropic
  conductivitie via an ad hoc misfit functional}.
\newblock {\em Inverse Problems}, 37:125007, 2021.

\bibitem{foschiatti2025recovering}
S.~Foschiatti, R.~Gaburro, and E.~Sincich.
\newblock {Recovering a layered anisotropic admittivity in Calder\'on's problem
  with local measurements}.
\newblock {\em accepted in SIAM Journal on Mathematical Analysis}, 2025.

\bibitem{francini2023propagation}
E.~Francini, S.~Vessella, and J.-N. Wang.
\newblock {Propagation of smallness and size estimate in the second order
  elliptic equation with discontinuous complex Lipschitz conductivity}.
\newblock {\em Journal of Differential Equations}, 343:687--717, 2023.

\bibitem{G-L2009}
R.~Gaburro and W.~R.~B. Lionheart.
\newblock {Recovering Riemannian metrics in monotone families from boundary
  data}.
\newblock {\em Inverse Problems}, 25:045004, 2009.

\bibitem{G-S2015}
R.~Gaburro and E.~Sincich.
\newblock Lipschitz stability for the inverse conductivity problem for a
  conformal class of anisotropic conductivities.
\newblock {\em Inverse Problems}, 31:015008, 2015.

\bibitem{gilbarg1977elliptic}
D.~Gilbarg and N.~S. Trudinger.
\newblock {\em Elliptic partial differential equations of second order}, volume
  224.
\newblock Springer, 1977.

\bibitem{haberman2015uniqueness}
B.~Haberman.
\newblock {Uniqueness in Calder{\'o}n’s problem for conductivities with
  unbounded gradient}.
\newblock {\em Communications in Mathematical Physics}, 340:639--659, 2015.

\bibitem{haberman2013uniqueness}
B.~Haberman and D.~Tataru.
\newblock {Uniqueness in Calder{\'o}n’s problem with Lipschitz
  conductivities}.
\newblock {\em Duke Mathematical Journal}, 162:497 -- 516, 2013.

\bibitem{harrach2009uniqueness}
B.~Harrach.
\newblock On uniqueness in diffuse optical tomography.
\newblock {\em Inverse problems}, 25(5):055010, 2009.

\bibitem{harrach2012simultaneous}
B.~Harrach and M.~Lassas.
\newblock Simultaneous determination of the diffusion and absorption
  coefficientfrom boundary data.
\newblock {\em Inverse Problems \& Imaging}, 6(4), 2012.

\bibitem{kim2024neumann}
S.~Kim and G.~Sakellaris.
\newblock {The Neumann Green function and scale-invariant regularity estimates
  for elliptic equations with Neumann data in Lipschitz domains}.
\newblock {\em Calculus of Variations and Partial Differential Equations},
  63(8):219, 2024.

\bibitem{kohn1984determining}
R.~Kohn and M.~Vogelius.
\newblock Determining conductivity by boundary measurements.
\newblock {\em Communications on pure and applied mathematics}, 37(3):289--298,
  1984.

\bibitem{kohn1984identification}
R.~V. Kohn and M.~Vogelius.
\newblock Identification of an unkown conductivity by menas of measurements at
  the boundary.
\newblock In {\em SIAM-AMS Proceedings}, pages 113--123. American Mathematical
  Soc, 1984.

\bibitem{lassas2001determining}
M.~Lassas and G.~Uhlmann.
\newblock {On determining a Riemannian manifold from the Dirichlet-to-Neumann
  map}.
\newblock In {\em Annales scientifiques de l'Ecole normale sup{\'e}rieure},
  volume~34, pages 771--787, 2001.

\bibitem{lassas2003dirichlet}
M.~Lassas, G.~Uhlmann, and M.~Taylor.
\newblock {The Dirichlet-to-Neumann map for complete Riemannian manifolds with
  boundary}.
\newblock {\em Communications in Analysis and Geometry}, 11(2):207--221, 2003.

\bibitem{Lee-U1989}
J.~M. Lee and G.~Uhlmann.
\newblock Determining anisotropic real-analytic conductivities by boundary
  measurements.
\newblock {\em Commun. Pure Appl. Math.}, 42:1097--1112, 1989.

\bibitem{li2003estimates}
Y.~Li and L.~Nirenberg.
\newblock Estimates for elliptic systems from composite material.
\newblock {\em Communications on Pure and Applied Mathematics}, 56(7):892--925,
  2003.

\bibitem{L1997}
W.~R.~B. Lionheart.
\newblock Conformal uniqueness results in anisotropic electrcal impedance
  imaging.
\newblock {\em Inverse Problems}, 13:125--134, 1997.

\bibitem{miranda2013partial}
C.~Miranda.
\newblock {\em Partial differential equations of elliptic type}.
\newblock Springer-Verlag, 2013.

\bibitem{mitrea2000potential}
M.~Mitrea and M.~Taylor.
\newblock {Potential theory on Lipschitz domains in Riemannian manifolds:
  H\"older continuous metric tensors}.
\newblock {\em Communications in Partial Differential Equations},
  25(7-8):1487--1536, 2000.

\bibitem{nachman1996global}
A.~I. Nachman.
\newblock Global uniqueness for a two-dimensional inverse boundary value
  problem.
\newblock {\em Annals of Mathematics}, pages 71--96, 1996.

\bibitem{sylvester1990anisotropic}
J.~Sylvester.
\newblock An anisotropic inverse boundary value problem.
\newblock {\em Communications on Pure and Applied Mathematics}, 43(2):201--232,
  1990.

\bibitem{sylvester1987global}
J.~Sylvester and G.~Uhlmann.
\newblock A global uniqueness theorem for an inverse boundary value problem.
\newblock {\em Annals of Mathematics}, pages 153--169, 1987.

\bibitem{uhlmann2009electrical}
G.~Uhlmann.
\newblock {Electrical impedance tomography and Calder\'on's problem}.
\newblock {\em Inverse problems}, 25(12), 2009.

\end{thebibliography}

\end{document}